\documentclass[reqno,11pt]{article}
\usepackage{a4wide}
\usepackage{color,eucal,enumerate,mathrsfs, amsthm}
\usepackage[normalem]{ulem}
\usepackage{amsmath,amssymb,epsfig,bbm}
\numberwithin{equation}{section}
\usepackage{tikz}

\usepackage[pdfborder={0 0 0}]{hyperref}  

\usepackage[latin1]{inputenc}
\usepackage[english]{babel}

\newcommand{\E}{{\sf E}}

\renewcommand{\H}{\mathscr{H}}
\newcommand{\M}{\mathcal{M}}
\newcommand{\NN}{\mathcal{N}}
\newcommand{\N}{\mathbb{N}}

\newcommand{\R}{\mathbb{R}}
\newcommand{\Z}{\mathbb{Z}}

\newcommand{\mm}{{\mbox{\boldmath$m$}}}

\newcommand{\ggamma}{{\mbox{\boldmath$\gamma$}}}

\newcommand{\ppi}{{\mbox{\boldmath$\pi$}}}

\newcommand{\Ggamma}{{\mbox{\boldmath$\Gamma$}}}

\newcommand{\sfd}{{\sf d}}

\newcommand{\sfh}{{\sf h}}

\newcommand{\Kliminf}{K\kern-3pt-\kern-2pt\mathop{\rm lim\,inf}\limits}  
\newcommand{\Lip}{\mathop{\rm Lip}\nolimits}          
\renewcommand{\d}{{\mathrm d}}

\newcommand{\restr}[1]{\lower3pt\hbox{$|_{#1}$}}
\newcommand{\la}{\left<}                  
\newcommand{\ra}{\right>}
\newcommand{\eps}{\varepsilon}  
\newcommand{\nchi}{{\raise.3ex\hbox{$\chi$}}}

\newcommand{\limi}{\varliminf}
\newcommand{\lims}{\varlimsup}

\newcommand{\fr}{\penalty-20\null\hfill$\blacksquare$}                      

\newcommand{\prob}[1]{\mathscr P(#1)}                   
\newcommand{\probt}[1]{\mathscr P_2(#1)}                   
\newcommand{\e}{{\rm{e}}}                           

\renewcommand{\mm}{\mathfrak m}                                

\renewenvironment{proof}{\removelastskip\par\medskip   
\noindent{\em proof} \rm}{\penalty-20\null\hfill$\square$\par\medbreak}
\newenvironment{sketch}{\removelastskip\par\medskip   
\noindent{\em Sketch of the proof} \rm}{\penalty-20\null\hfill$\square$\par\medbreak}

\newtheorem{theorem}{Theorem}[section]
\newtheorem{corollary}[theorem]{Corollary}
\newtheorem{lemma}[theorem]{Lemma}
\newtheorem{proposition}[theorem]{Proposition}
\newtheorem{definition}[theorem]{Definition}

\newtheorem{example}[theorem]{Example}
\newtheorem{remark}[theorem]{Remark}
\newtheorem{thmdef}[theorem]{Theorem/Definition}

\newcommand{\bd}{{\mathbf\Delta}}

\newcommand{\s}{{\rm S}}

\newcommand{\test}[1]{{\rm Test}(#1)}
\newcommand{\vsm}{{\rm TestV}(\X)}
\newcommand{\ffsm}[1]{{\rm TestForm}_{#1}(\X)}
\newcommand{\Int}[1]{{\rm Int}(#1)}

\newcommand{\X}{{\sf X}}
\newcommand{\Y}{{\rm Y}}
\renewcommand{\Z}{{\rm Z}}

\newcommand{\h}{{\sfh}}

\renewcommand{\ae}{{\textrm{\rm{-a.e.}}}}

\newcommand{\RCD}{{\sf RCD}}
\newcommand{\comp}{{\rm Comp}}
\newcommand{\closed}[1]{{\rm C}_{#1}(\X)}
\newcommand{\exact}[1]{{\rm E}_{#1}(\X)}
\newcommand{\exacto}[1]{\overline{\rm E}_{#1}(\X)}

\newcommand{\lf}{{\sf Lip}}
\newcommand{\cf}{{\sf Comp}}
\newcommand{\ric}{{\mathbf{Ric}}}

\newcommand{\harm}[1]{{\rm Harm}_{#1}(\X)}
\newcommand{\mes}{{\sf Meas}}
\newcommand{\PCM}{{\sf Pcm}}
\renewcommand{\Int}{{\sf Int}}
\newcommand{\LIP}{{\rm LIP}}
\newcommand{\Der}{{\sf Der}}

\renewcommand{\S}{{\rm S}}

\newcommand{\lip}{{\rm lip}}
\newcommand{\HS}{{\rm HS}}
\renewcommand{\div}{{\rm div}}
\newcommand{\He}{{\rm Hess}}
\newcommand{\Ho}{{\rm H}}

\newcommand{\eh}{\sf E_\Ho}
\newcommand{\ec}{\sf E_C}

\newcommand{\weakgrad}[1]{|{\rm D} #1|}

\DeclareMathOperator*{\esssup}{\rm ess-sup}

\title{Lecture notes on differential calculus on $\RCD$ spaces}
\begin{document}

\author{
   Nicola Gigli\
   \thanks{SISSA. email: \textsf{ngigli@sissa.it}}
   }

\maketitle




\tableofcontents

\section*{Introduction}
These are extended notes of the course given by the author at RIMS, Kyoto, in October 2016. The aim is to give a self-contained  overview on the recently developed approach to differential calculus on metric measure spaces, with most, but not all, the material coming from \cite{Gigli14}. The effort is directed into giving as many ideas as possible, without losing too much time in technical details and utmost generality: for this reason many statements are  given under some simplifying assumptions and proofs are sometimes only sketched.

The notes are divided in two parts: in the first one we study the first-order structure of general metric measure spaces, then, building on top of this,  in the second we study the second-order differential structure of spaces with (Riemannian) Ricci curvature bounded from below.

\bigskip

For what concerns the first part, a crucial role is played by the concept of $L^2$-normed $L^\infty$-module, which provides a convenient abstraction of the notion of `space of $L^2$ sections of a vector bundle'. This is a variant of the similar notion of $L^\infty$-module introduced by Weaver in \cite{Weaver01} who was also interested in developing a calculus on non-smooth spaces. In fact, some of the statements which we shall present in Sections \ref{se:cot} and \ref{se:tan} can be seen as technical variants of analogous statements given in \cite{Weaver01}. Still, our axiomatization and the study of Sobolev functions carried out in \cite{AmbrosioGigliSavare11} allow to produce new and interesting links between the abstract differential calculus and the structure of the space: for instance, in Theorem \ref{thm:speedplan} we shall see that we can associate to `almost every absolutely continuous curve' a derivative whose modulus coincides with the metric speed of the curve itself. This kind of statement, whose  precise formulation requires the notions of `test plan' and of `pullback of a module', is crucial in applications to geometry, see for instance \cite{DPG16}.

We also remark that the definition of cotangent module that we give here can be canonically identified with the cotangent bundle as built by Cheeger in \cite{Cheeger00}. We won't insist on this point (referring to \cite{Gigli14} for more details) because the two approaches are very different in spirit: in \cite{Cheeger00}, working on doubling spaces supporting a Poincar\'e inequality, Cheeger gave  a metric version of Rademacher's theorem, which results in  much more than a mere definition of cotangent bundle. Here, instead, we are only interested in giving an abstract and weak notion of differential of a Sobolev function and we shall do so without imposing any doubling or Poincar\'e inequality. In any case, our first-order theory should mostly be regarded as foundational material for the second-order one on $\RCD$ spaces.

\bigskip

In the second part of the notes we shall work in $\RCD$ spaces, mostly without imposing any dimension bound (we  confine to the final Section \ref{se:findim} some recent results about calculus on finite dimensional spaces). The definition of $\RCD(K,\infty)$ spaces that we shall adopt  is the one, coming from \cite{AmbrosioGigliSavare12}, based on the appropriate weak formulation of the Bochner inequality
\begin{equation}
\label{eq:bocintro1}
\Delta\frac{|\nabla f|^2}2\geq \langle\nabla f,\nabla \Delta f\rangle+K|\nabla f|^2.
\end{equation}

There  is a certain amount of `cheating' in choosing this approach, because it is the closest to differential calculus and the furthest from the fact, crucial for the theory, that the class of $\RCD(K,\infty)$ spaces is closed w.r.t.\ measured-Gromov-Hasdorff convergence. Nevertheless, the validity of Bochner inequality on $\RCD$ spaces is now well-established within the theory, so that possibly there is no much harm in taking it as starting point for our discussion. The reader interested in the stability issue might want to start from the lecture notes \cite{AmbrosioGigliSavare-compact} for an account of the path which starts from the original approach of Lott-Sturm-Villani (\cite{Lott-Villani09}, \cite{Sturm06I}) and uses the heat flow (\cite{Gigli10}, \cite{Gigli-Kuwada-Ohta10}, \cite{AmbrosioGigliSavare11}) to isolate `Riemannian' spaces (\cite{AmbrosioGigliSavare11-2}) by also providing a stable version of the Bochner inequality (\cite{AmbrosioGigliSavare12}).

From the technical point of view, the main result of this second part of the notes (Lemmas \ref{le:lemmachiave} and \ref{le:riscritto}) is the improvement of the Bochner inequality from \eqref{eq:bocintro1} to:
\begin{equation}
\label{eq:bocintro2}
\Delta\frac{|X|^2}2\geq |\nabla X|_{\HS}^2-\langle X,(\Delta_\Ho X^\flat)^\sharp\rangle+K|X|^2
\end{equation}
in the appropriate weak sense. Notice that for $X=\nabla f$, \eqref{eq:bocintro2} reduces to \eqref{eq:bocintro1} with the additional non-negative contribution $|\He f|_\HS^2$ on the right hand side. Here the language of $L^2$-normed modules provides natural spaces where objects like the Hessian or the covariant derivative belong, and one of the effects of the improved formula \eqref{eq:bocintro2} is the bound
\begin{equation}
\label{eq:bocintro3}
\int|\He f|_\HS^2\,\d\mm\leq \int (\Delta f)^2-K|\nabla f|^2\,\d\mm
\end{equation}
obtained integrating  \eqref{eq:bocintro2} for $X=\nabla f$ (Corollary \ref{cor:bello}). Since functions with gradient and Laplacian in $L^2$ are easy to build using the heat flow, \eqref{eq:bocintro3} grants that there are `many' functions with Hessian in $L^2$. Starting from this, it will not be hard to build a second order calculus and an indication of the novelty of the theory is in the fact that we can prove that the exterior differential is a closed operators on the space of $k$-forms for any $k\in\N$ (Theorem \ref{thm:basew12d}), whereas previously known results only covered the case $k=0$ (\cite{Cheeger00}, \cite{Weaver01}, \cite{Cheeger-Colding97III}). In particular, quite natural versions of the De Rham cohomology and of the Hodge theorem can be provided (Section \ref{se:dr})

Another consequence of the fact that we have well-defined differential operators is that we can define the Ricci curvature as the quantity for which the Bochner identity holds:
\[
\ric(X,X):=\Delta\frac{|X|^2}2- |\nabla X|_{\HS}^2+\langle X,(\Delta_\Ho X^\flat)^\sharp\rangle.
\]
It turns out that $\ric (X,X)$ is a measure-valued tensor and the role of \eqref{eq:bocintro2} is to grant that the Ricci curvature is bounded from below by $K$, as expected.

\bigskip

Finally, a feature of the language proposed here is that the differential operators are stable w.r.t.\ measured-Gromov-Hausdorff convergence of the base spaces in a quite natural sense. To keep the presentation short we won't discuss this - important and under continuous development - topic, referring to \cite{H15}, \cite{AST17}, \cite{AH16} for recent results.

\bigskip

{\bf Acknowledgment} I wish to thank RIMS for the invitation in giving a course there and the very warm hospitality. This project has also been partly financed by the MIUR SIR-grant `Nonsmooth Differential Geometry' (RBSI147UG4).

\section{First order theory for general metric measure spaces}
\subsection{Sobolev functions on metric measure spaces}
For the purpose of this note a metric measure space  $(\X,\sfd,\mm)$ is a complete separable metric space $(\X,\sfd)$ endowed with a non-negative (and not zero) Borel measure $\mm$ giving finite mass to bounded sets.

$\prob \X$ is the space of Borel probability measures on $\X$ and $C([0,1],\X)$   the space of continuous curves with value in $\X$ endowed with the $\sup$ norm. For $t\in[0,1]$ the evaluation map $\e_t:C([0,1],\X)\to \X$ is  defined by
\[
\e_t(\gamma):=\gamma_t,\qquad\forall \gamma\in C([0,1],\X).
\]
Recall that $\gamma:[0,1]\to \X$ is absolutely continuous provided there is  $f\in L^1(0,1)$ such that
\begin{equation}
\label{eq:accurve}
\sfd(\gamma_t,\gamma_s)\leq \int_t^sf(r)\,\d r,\qquad\forall t,s\in[0,1],\ t<s.
\end{equation}
In this case, for a.e.\ $t\in[0,1]$ there exists  $|\dot\gamma_t|:=\lim_{h\to 0}\frac{\sfd(\gamma_{t+h},\gamma_t)}{|h|}$ and $|\dot\gamma_t|$ is the least, in the a.e.\ sense, function $f\in L^1(0,1)$ for which \eqref{eq:accurve} holds (see e.g.\ Theorem 1.1.2 of \cite{AmbrosioGigliSavare08} for a proof).

By $\LIP(\X)$ (resp. $\LIP_b(\X)$) we mean the space of Lipschitz (resp. Lipschitz and bounded) functions on $\X$.

\bigskip

There are several equivalent definitions of Sobolev functions on a metric measure space (\cite{Cheeger00}, \cite{Shanmugalingam00}, \cite{AmbrosioGigliSavare11}), here we shall adopt one of those proposed in the latter reference, where the notion of Sobolev function is given in duality with that of test plan:
\begin{definition}[Test Plans]
Let $\ppi\in\prob{C([0,1], \X)}$. We say that $\ppi$ is a test plan provided for some $C>0$ we have
\[
\begin{split}
(\e_t)_*\ppi&\leq C \mm,\qquad\forall t\in[0,1],\\
\iint_0^1|\dot\gamma_t|^2\,\d t\,\d\ppi(\gamma)&<\infty.
\end{split}
\]
The least such $C$ is called compression constant of $\ppi$ and denoted as  $\cf(\ppi)$.
\end{definition}
Recall that $L^0(\X)$ is the space of (equivalence classes w.r.t.\ $\mm$-a.e.\ equality of) Borel real valued functions on $\X$.
\begin{definition}[The Sobolev class $\s^2(\X,\sfd, \mm)$]
The Sobolev class $\s^2(\X,\sfd,\mm)$, or simply $\s^2(\X)$ is the space of all functions $f\in L^0(\X)$ such that there exists a non-negative $G\in L^2( \mm)$, called weak upper gradient of $f$,  for which it holds
\begin{equation}
\label{eq:defsob}
\int|f(\gamma_1)-f(\gamma_0)|\,\d\ppi(\gamma)\leq \iint_0^1G(\gamma_t)|\dot\gamma_t|\,\d t\,\d\ppi(\gamma),\qquad\forall \ppi\textrm{ test plan}.
\end{equation}
\end{definition}
Notice that the assumptions on $\ppi$ grant that the integrals are well defined and that the one in the right hand side is finite. With an argument based on the stability of the class of test plans by `restriction' and `rescaling' it is not hard to check that $f\in\s^2(\X)$ with $G$ being a weak upper gradient if and only if for any test plan $\ppi$ and  any $t,s\in[0,1]$, $t<s$ it holds
\begin{equation}
\label{eq:localsob}
|f(\gamma_s)-f(\gamma_t)|\leq \int_t^sG(\gamma_r)|\dot\gamma_r|\,\d r\qquad\ppi\ae\ \gamma.
\end{equation}
Then an application of Fubini's theorem (see  \cite{Gigli12} for the details) shows that this is in turn equivalent to: for any test plan $\ppi$ and $\ppi$-a.e.\ $\gamma$, the function $t\mapsto f(\gamma_t)$ is in $W^{1,1}(0,1)$ and
\begin{equation}
\label{eq:localsob2}
\big|\frac\d{\d t}f(\gamma_t)\big|\leq G(\gamma_t)|\dot\gamma_t|,\qquad{\rm a.e. }\ t.
\end{equation}
It is then easy to check that  there exists a minimal $G$ in the $\mm$-a.e. sense for which \eqref{eq:defsob} holds: such $G$ will be called {\bf minimal weak upper gradient} and denoted by $\weakgrad f$.

From the definitions it is clear that $\S^2(\X)$ is a vector space and that
\begin{equation}
\label{eq:s2vector}
\weakgrad{(\alpha f+\beta g)}\leq |\alpha|\weakgrad f+|\beta|\weakgrad g\qquad\forall f,g\in\s^2(\X),\ \alpha,\beta\in\R.
\end{equation}
Beside this, the two crucial properties of minimal weak upper gradients that we shall use are:

\noindent\underline{Lower semicontinuity of minimal weak upper gradients}. Let $(f_n)\subset \s^2(\X)$ and $f\in L^0(\X)$ be such that $f_n\to f$ as $n\to\infty$ in $L^0(\X)$ (i.e.\ $\mm$-a.e.). Assume that $(\weakgrad {f_n})$ converges to some $G\in L^2(\X)$ weakly in $L^2(\X)$.

Then 
\begin{equation}
\label{eq:lscwug}
f\in \s^2(\X )\qquad\text{ and }\qquad\weakgrad f\leq  G,\quad \mm\ae.
\end{equation}

\noindent\underline{Locality}. The minimal weak upper gradient is local in the following sense:
\begin{equation}
\label{eq:localgrad0}
\weakgrad f=0,\qquad\mm\ae\textrm{  on }\{f=0\},\qquad \forall f\in\s^2(\X).
\end{equation}
\eqref{eq:lscwug}  follows quite easily from the very definition of $\S^2(\X)$, while \eqref{eq:localgrad0} comes from the characterization \eqref{eq:localsob2} and the analogous property of functions in $W^{1,1}(0,1)$.

The lower semicontinuity of minimal weak upper gradients ensures that the space $W^{1,2}(\X):= L^2\cap \s^2(\X)$ endowed with the norm
\[
\|f\|_{W^{1,2}(\X)}^2:=\|f\|_{L^2(\X)}^2+\|\weakgrad f\|^2_{L^2(\X)}.
\]
is a Banach space. It is trivial to check that Lipschitz functions with bounded support are in $W^{1,2}(\X)$ with 
\[
\weakgrad f\leq \lip(f)\qquad\mm\ae,
\]
where
\[
\lip(f)(x):=\lims_{y\to x}\frac{|f(y)-f(x)|}{\sfd(x,y)}\quad\text{ if $x$ is not isolated,\qquad 0 otherwise.}
\]
In particular,   $W^{1,2}(\X)$ is dense in $L^2(\X)$. On the other hand it is non-trivial that for every  $f\in W^{1,2}(\X)$ there exists a sequence $(f_n)$ of Lipschitz functions with bounded support   converging to $f$ in $L^2$ such that
\[
\int \weakgrad f^2\,\d\mm=\lim_n\int \lip^2(f_n)\,\d\mm.
\]
We shall not use this fact (see \cite{AmbrosioGigliSavare11} for the proof).

We conclude recalling that, as shown in \cite{ACM14}, 
\begin{equation}
\label{eq:refsep}
\text{if $W^{1,2}(\X)$ is reflexive, then it is separable.}
\end{equation}
This can be proved considering a countable $L^2$-dense set $D$ of the unit ball $B$ of $W^{1,2}(\X)$. Then for $f\in B$ find $(f_n)\subset D$ converging to $f$ in $L^2(\X)$: being $(f_n)$ bounded in $W^{1,2}(\X)$, up to subsequences it must have a weak limit in $W^{1,2}(\X)$ and this weak limit must be $f$. Hence the weak closure of $D$ is precisely $B$ and by Mazur's lemma this is sufficient to conclude. 

\subsection{$L^2$-normed modules, cotangent module and differential}\label{se:cot}
\subsubsection{$L^2$-normed modules}

\begin{definition}[$L^2(\X)$-normed $L^\infty(\X)$-modules]
A $L^2(\X)$-normed $L^\infty(\X)$-module, or simply a $L^2(\X)$-normed module, is a structure $(\M,\|\cdot\|,\cdot,|\cdot|)$ where
\begin{itemize}
\item[i)] $(\M,\|\cdot\|)$ is a Banach space
\item[ii)] $\cdot$ is a bilinear map from $L^\infty(\X)\times \M$ to $\M$, called multiplication by $L^\infty(\X)$ functions, such that
\begin{subequations}
\label{eq:moltf}
\begin{align}
\label{eq:m0}f\cdot(g\cdot v)&=(fg)\cdot v,\\
\label{eq:m1}{\mathbf 1}\cdot v&=v,
\end{align}
\end{subequations}
for every $v\in \M$ and $f,g\in L^\infty(\X)$, where ${\bf 1}$ is the function identically equal to 1.
\item[iii)] $|\cdot|$ is a map from $\M$ to $L^2(\X)$, called pointwise norm, such that 
\begin{subequations}
\label{eq:pontnorm}
\begin{align}
\label{eq:normgeq} |v|&\geq 0\qquad\mm\ae\\
\label{eq:normpunt} |fv|&=|f|\,|v|\qquad\mm\ae\\
\label{eq:recnorm} \|v\|&=\sqrt{\int|v|^2\,\d\mm},
\end{align}
\end{subequations}
\end{itemize}
An isomorphism between  two $L^2(\X)$-normed modules is a linear bijection which preserves the norm, the product with $L^\infty(\X)$ functions and the pointwise norm.
\end{definition}
We shall typically write $fv$ in place of $f\cdot v$ for the product with $L^\infty(\X)$ function.

Notice that thanks to \eqref{eq:m1}, for $\lambda\in \R$ and $v\in \M$ the values of $\lambda v$ intended as coming from the vector space structure and as the product with the function constantly equal to $\lambda$ agree, so that the expression is unambiguous. Also, from \eqref{eq:normpunt} and \eqref{eq:recnorm} we obtain that
\[
\|fv\|\leq\|f\|_{L^\infty}\|v\|.
\]
We also remark that  the pointwise norm  satisfies
\[
\begin{split}
|\lambda v|&=|\lambda|\, |v|\\
|v+w|&\leq |v|+|w|,
\end{split}
\]
$\mm$-a.e.\ for every $v,w\in \M$ and $\lambda\in\R$. Indeed, the first comes from  \eqref{eq:normpunt}, while for the second we argue by contradiction. If it were false, for some $v,w\in \M$, Borel set $E\subset \X$ with $\mm(E)\in(0,\infty)$ and positive real numbers $a,b,c$ with $a+b<c$ we would have  $\mm$-a.e.\ on $E$
\[
|v+w|\geq c\qquad|v|\leq a\qquad |w|\leq b
\]
However, this creates a contradiction with \eqref{eq:recnorm} and the fact that $\|\cdot\|$ is a norm because 
\[
\begin{split}
\|\nchi_Ev\|+\|\nchi_Ew\|&=\|\nchi_E|v|\|_{L^2}+\|\nchi_E|w|\|_{L^2}\leq \sqrt{\mm(E)}\,(a+b)\\
&<\sqrt{\mm(E)}\,c\leq \|\nchi_E|v+w|\|_{L^2}=\|\nchi_E(v+w)\|.
\end{split}
\]
In the following for given $v,w\in\M$ and Borel set $E\subset \X$ we shall say that $v=w$  $\mm$-a.e.\ on $E$,  provided
\[
\nchi_{E}(v-w)=0\qquad\text{or equivalently if }\qquad |v-w|=0\quad\mm\ae\text{ on }E.
\]
\begin{example}{\rm Consider a manifold $\X$ equipped with a reference measure $\mm$ and with a normed vector bundle. Then the space of $L^2(\X,\mm)$-sections of the bundle naturally carries  the structure of $L^2(\X)$-normed module. This is the example which motivates the abstract definition of $L^2(\X)$-normed module.
}\fr\end{example}
We say that $f\in L^\infty(\X)$ is {\bf simple} provided it attains only a finite number of values.
 \begin{definition}[Generators]
We say that $V\subset \M$ generates $\M$ provided finite sums of the form $\sum_i\nchi_{E_i}v_i$ with $(E_i)$ Borel partition of $X$ and $(v_i)\subset V$ are dense in $\M$.
\end{definition}
By approximating $L^\infty$ functions with simple ones, it is easy to see that $V$ generates $\M$ if and only if $L^\infty$-linear combinations of elements of $V$ are dense in $\M$.\bigskip

A particularly important class of modules is that of {\bf Hilbert modules}, i.e.\ modules $\H$ which are, when seen as Banach spaces, Hilbert spaces. It is not hard to check that in this case the pointwise norm satisfies the pointwise parallelogram identity
\[
|v+w|^2+|v-w|^2=2(|v|^2+|w|^2)\quad\mm\ae\qquad\forall v,w\in\H
\]
and thus that by polarization it induces a pointwise scalar product $\la\cdot,\cdot\ra:\H^2\to L^1(\X)$ which is $L^\infty(\X)$-bilinear and satisfies
\[
\begin{split}
|\langle v,w\rangle|\leq |v|\,|w|\qquad\qquad\qquad\la v,v\ra=|v|^2, 
\end{split}
\]
$\mm$-a.e.\ for every $v,w\in\H$.

\bigskip

It is at times convenient to deal with objects with less integrability; in this direction, the following concept is useful:
\begin{definition}[$L^0$-normed module] A $L^0$-normed module is a structure $(\M,\tau,\cdot,|\cdot|)$ where:
\begin{itemize}
\item[i)] $\cdot$ is a bilinear map, called multiplication with $L^0$ functions, from $L^0(\X)\times \M$ to $\M$ for which \eqref{eq:m0}, \eqref{eq:m1} hold for any $f\in L^0(\X)$, $v\in \M$,
\item[ii)] $|\cdot|:\M\to L^0(\X)$, called pointwise norm, satisfies \eqref{eq:normgeq} and \eqref{eq:normpunt}  for any $f\in L^0(\X)$, $v\in \M$,
\item[iii)] for some Borel partition $(E_i)$ of $\X$ into sets of finite $\mm$-measure, $\M$ is complete w.r.t.\ the distance 
\begin{equation}
\label{eq:m0dist}
\sfd_0(v,w):=\sum_i\frac{1}{2^i\mm(E_i)}\int_{E_i}\min\{1,|v-w|\}\,\d\mm
\end{equation}
and $\tau$ is the topology induced by the distance.
\end{itemize}
An isomorphims of $L^0$-normed modules is a linear homeomorphism preserving the pointwise norm and the multiplication with $L^0$-functions.
\end{definition}
It is readily checked that the choice of the partition $(E_i)$ in $(iii)$ does not affect the completeness of $\M$ nor the topology $\tau$.

\begin{thmdef}[$L^0$ completion of a module]\label{thm:defl0} Let $\M$ be a $L^2$-normed module. Then there exists a unique couple $(\M^0,\iota)$, where $\M^0$ is a $L^0$-normed module and $\iota:\M\to \M^0$ is linear, preserving the pointwise norm and  with dense image. 

Uniqueness is intended up to unique isomorphism, i.e.:  if $(\tilde \M^0,\tilde\iota)$ has the same properties, then there exists a unique isomorphism $\Phi:\M^0\to \tilde \M^0$ such that $\tilde\iota=\Phi\circ\iota$.
\end{thmdef}
\begin{proof}Uniqueness is trivial.  For existence define $\M^0$ to be the metric completion of $\M$ w.r.t.\ the distance defined in \eqref{eq:m0dist} and $\iota$ as the natural embedding, then observe that the $L^2$-normed module structure of $\M$ can be extended by continuity and induce an $L^0$-normed module structure on $\M^0$. 
\end{proof}
\subsubsection{Cotangent module and differential}
The cotangent module $L^2(T^*\X)$ and the differential $\d:\S^2(\X)\to L^2(T^*\X)$ are defined, up to unique isomorphism, by the following theorem. The elements of the cotangent module will be called 1-forms.
\begin{thmdef}\label{thm:defcot}
There exists a unique couple $(L^2(T^*\X),\d)$ with $L^2(T^*\X)$ being a $L^2$-normed module and $\d:\S^2(\X)\to L^2(T^*\X)$ a linear map such that:
\begin{itemize}
\item[i)] for any $f\in \S^2(\X)$ it holds $|\d f|=\weakgrad f$ $\mm$-a.e.,
\item[ii)] $L^2(T^*\X)$ is generated by $\{\d f:f\in \S^2(\X)\}$.
\end{itemize}
Uniqueness is intended up to unique isomorphism, i.e.: if $(\M,\d')$ is another such couple, then there is a unique isomorphism $\Phi:L^2(T^*\X)\to \M$ such $\Phi(\d f)=\d' f$ for every $f\in\s^2(\X)$.
\end{thmdef}
Note: we shall call a form $\omega\in L^2(T^*\X)$ {\bf simple} if it can be written as $\sum_i\nchi_{A_i}\d f_i$ for a finite Borel partition $(A_i)$ of $\X$ and $(f_i)\subset \S^2(\X)$.
\begin{proof}\\*
{\bf Uniqueness} Consider a simple form $\omega\in L^2(T^*\X)$ and notice that the requirements that $\Phi$ is $L^\infty$-linear and such that  $\Phi(\d f)=\d' f$ force the definition
\begin{equation}
\label{eq:defPhi}
\Phi(\omega):=\sum_i\nchi_{A_i}\d'f_i\qquad\text{for}\qquad \omega=\sum_i\nchi_{A_i}\d f_i.
\end{equation}
The identity 
\[
|\Phi(\omega)|=\sum_i\nchi_{A_i}|\d'f_i|\quad \stackrel{(i)\text{ for }\M}  =\quad\sum_i\nchi_{A_i}\weakgrad{f_i}\quad \stackrel{(i)\text{ for }L^2(T^*\X)} =\quad\sum_i\nchi_{A_i}|\d f_i|=|\omega|
\]
shows in particular that the definition of $\Phi(\omega)$ is well-posed, i.e.\ $\Phi(\omega)$ depends only on $\omega$ and not on the way we represent it as finite sum. It also shows that $\Phi$ preserves the pointwise norm of simple forms and thus, since $\Phi$ is clearly linear, grants that $\Phi$ is continuous. Being simple forms dense in $L^2(T^*\X)$ (by property $(ii)$ for $L^2(T^*\X)$), $\Phi$ can be uniquely extended by continuity to a map from $L^2(T^*\X)$ to $\M$ and this map is clearly linear, continuous and preserves the pointwise norm. Also, from the very definition \eqref{eq:defPhi} we see that $\Phi(f\omega)=f\Phi(\omega)$ for simple $f$ and $\omega$, so that by approximation we see that the same holds for general $f\in L^\infty(\X)$, $\omega\in L^2(T^*\X)$. Property \eqref{eq:recnorm} grants that $\Phi$ also preserves the norm, so that to conclude it is sufficient to show that its image is the whole $\M$. This follows from the density of simple forms in $\M$ (property $(ii)$ for $\M$).

{\bf Existence} We define the `Pre-cotangent module' $\PCM$ to be the set of finite sequences $(A_i,f_i)$ with $(A_i)$ being a Borel partition of $\X$ and $(f_i)\subset \S^2(\X)$. Then we define an equivalence relation on $\PCM$ by declaring $(A_i,f_i)\sim (B_j,g_j)$ iff for every $i,j$ we have
\[
\weakgrad{(f_i-g_j)}=0,\qquad \mm\ae\text{ on }\ A_i\cap B_j.
\]
Denoting by $[A_i,f_i]$ the equivalence class of $(A_i,f_i)$, we endow $\PCM/\sim$ with a vector space structure by putting
\[
\begin{split}
[A_i,f_i]+[B_j,g_j]&:=[A_i\cap B_j,f_i+g_j],\\
\lambda[A_i,f_i]&:=[A_i,\lambda f_i].
\end{split}
\]
Notice that thanks to the locality property \eqref{eq:localgrad0} of the minimal weak upper gradient, these definitions are well posed. For the same reason, the quantity
\[
\|[A_i,f_i]\|:=\sqrt{\sum_i\int_{A_i}\weakgrad{f_i}^2\,\d\mm}
\]
is well defined, and from \eqref{eq:s2vector} we see that it is a norm. Let $(L^2(T^*\X),\|\cdot\|)$ be the completion of $(\PCM/\sim,\|\cdot\|)$ and $\d:\S^2(\X)\to L^2(T^*\X)$ be the map sending $f$ to $[\X,f]$. By construction, $L^2(T^*\X)$ is a Banach space and $\d$ is linear. We want to endow $L^2(T^*\X)$ with the structure of $L^2(\X)$-normed module and to this aim we define $|\cdot|:\PCM/\sim\to L^2(\X)$ by
\[
|[A_i,f_i]|:=\sum_i\nchi_{A_i}\weakgrad{f_i}
\]
and a bilinear map $\{\text{simple functions}\}\times \PCM/\sim\ \to\ \PCM/\sim$ by 
\[
\Big(\sum_j\alpha_j\nchi_{E_j}\Big)\cdot[A_i,f_i]:=[A_i\cap E_j,\alpha_jf_i],
\]
where $(E_j)$ is a finite partition of $\X$. It is readily verified that these definitions are well posed and that properties \eqref{eq:moltf} and \eqref{eq:pontnorm} hold for simple functions and elements of $\PCM/\sim$. It is also clear that $||\omega_1|-|\omega_2||\leq |\omega_1-\omega_2|$ $\mm$-a.e.\ for every $\omega_1,\omega_2\in \PCM/\sim$ and therefore we have
\[
\||\omega_1|-|\omega_2|\|_{L^2}\leq \|\omega_1-\omega_2\|,
\]
showing that the pointwise norm can, and will, be extended by continuity to the whole $L^2(T^*\X)$. Similarly, for $h:\X\to\R$ simple and $\omega\in \PCM/\sim$ from the identity $|h\omega|=|h||\omega|$ we obtain 
\[
\|h\omega\|^2=\int |h\omega|^2\,\d\mm\leq \|h\|^2_{L^\infty}\int |\omega|^2\,\d\mm= \|h\|^2_{L^\infty}\|\omega\|^2,
\]
showing that the multiplication by simple functions on $\PCM/\sim$ can, and will, be extended by continuity to a multiplication by $L^\infty(\X)$ functions on $L^2(T^*\X)$. 

The fact that properties \eqref{eq:moltf} and \eqref{eq:pontnorm} hold for these extensions follows trivially by approximation. Hence $L^2(T^*\X)$ is a $L^2(\X)$-normed module.

To conclude, notice that property $(i)$  is a direct consequence of the definition of $\d$ and of the pointwise norm. The fact that $L^2(T^*\X)$ is generated by $\{\d f:f\in \S^2(\X)\}$ also follows by the construction once we observe that the typical element $[A_i,f_i]$ of $\PCM/\sim$ is equal to $\sum_i\nchi_{A_i}\d f_i$ by the very definitions given.
\end{proof}
\begin{remark}\label{rem:gencot}{\rm
By a simple cut-off and truncation argument we see that $\{\d f:f\in W^{1,2}(\X)\}$ also generates $L^2(T^*\X)$. Hence, slightly more generally, we also have that if $D$ is a dense subset of $W^{1,2}(\X)$, then $\{\d f:f\in D\}$ generates $L^2(T^*\X)$.

This also shows that if $W^{1,2}(\X)$ is separable, then so is $L^2(T^*\X)$.
}\fr\end{remark}
\begin{remark}{\rm
It is not hard to check that if $\X$ is a smooth Finsler manifold, then $W^{1,2}(\X)$ as we defined it coincides with the Sobolev space defined via charts and that $\weakgrad f$ coincides a.e.\ with the norm of the distributional differential. 

From this fact and Theorem \ref{thm:defcot} it follows that the cotangent module can be identified with the space of $L^2$ sections of the cotangent bundle via the map which sends $\d f$ to the distributional differential of $f$.
}\fr\end{remark}
\begin{proposition}[Closure of the differential]\label{prop:closed} Let $(f_n)\subset \S^2(\X)$ be a sequence $\mm$-a.e. converging to some function $f\in L^0(\X)$. Assume that $(\d f_n)$ converges to some $\omega\in L^2(T^*\X)$ in the weak topology of $L^2(T^*\X)$ seen as Banach space. 

Then $f\in \S^2(\X)$ and $\d f=\omega$.
\end{proposition}
\begin{proof}
By applying Mazur's lemma we can assume that the convergence of $(\d f_n)$ to $\omega$ is strong in $L^2(T^*\X)$. In particular $(|\d f_n|)$ converges to $|\omega|$ in $L^2(\X)$ and by \eqref{eq:lscwug} this grants that $f\in \S^2(\X)$. For any $m\in\N$ we have that $f_n-f_m\to f-f_m$ $\mm$-a.e., thus using again \eqref{eq:lscwug}  we have 
\[
\|\d f-\d f_n\|_{L^2(T^*\X)}=\|\weakgrad{(f-f_n)}\|_{L^2(\X)}\leq\limi_m\|\weakgrad{(f_m-f_n)}\|_{L^2(\X)}=\limi_m\|\d f_m-\d f_n\|_{L^2(T^*\X)}
\]
and the conclusion follows letting $n\to\infty$ using the fact that, being $(\d f_n)$ strongly converging in $L^2(T^*\X)$, it is a Cauchy sequence.
\end{proof}

\begin{proposition}[Calculus rules]
The following holds.
\begin{itemize}
\item[-] \emph{Locality} For every $f,g\in \S^2(\X)$ we have
\begin{equation}
\label{eq:diffloc}
\d f=\d g\quad\mm\ae\ \text{\rm on} \ \{f=g\}.
\end{equation} 
\item[-] \emph{Chain rule} For every $f\in \S^2(\X)$ and $\varphi\in \LIP\cap C^1(\R)$  we have $\varphi\circ f\in \s^2(\X)$ and
\begin{equation}
\label{eq:chain}
\d(\varphi\circ f)=\varphi'\circ f\,\d f.
\end{equation}
\item[-] \emph{Leibniz rule} For every $f,g\in L^\infty\cap \S^2(\X)$ we have $fg\in \s^2(\X)$
\begin{equation}
\label{eq:leibniz}
\d(fg)=f\,\d g+g\,\d f.
\end{equation}
\end{itemize}
\end{proposition}
\begin{proof}\ \\
\noindent{\bf Locality} By the linearity of the differential the claim is equivalent to 
\[
\d f=0\qquad\mm\ae\text{ on} \ \{f=0\}
\]
which follows directly from $|\d f|=\weakgrad f$ $\mm$-a.e.\ and the locality property  \eqref{eq:localgrad0} of $\weakgrad f$.

\noindent{\bf Chain rule} The fact that  $\Lip(\varphi)\weakgrad f\in L^2(\X)$ is a weak upper gradient for $\varphi\circ f$ is obvious, hence in particular $\varphi\circ f\in\s^2(\X)$.

To prove  \eqref{eq:chain}, start noticing that taking into account the linearity of the differential and the fact that constant functions have 0 differential (because trivially their minimal weak upper gradient is 0), the chain rule \eqref{eq:chain} is trivial if $\varphi$ is affine. Hence by the locality property \eqref{eq:diffloc} the chain rule \eqref{eq:chain} holds if $\varphi$ is piecewise affine. Notice that this also forces  $\d f$ to be 0 $\mm$-a.e.\ on $f^{-1}(z)$ for any $z\in\R$, and thus also $\mm$-a.e.\ on $f^{-1}(\mathcal N)$ for $\mathcal N\subset\R$ countable.

Let now $\varphi\in\LIP\cap C^1(\R)$ and find a sequence $(\varphi_n)$ of equi-Lipschitz and piecewise affine functions such that $(\varphi_n),(\varphi_n')$ uniformly converge to  $\varphi,\varphi'$ respectively. From these, what previously said and the closure of the differential we can  pass to the limit in
\[
\d(\varphi_n\circ f)=\varphi'_n\circ f\,\d f
\]
and conclude.

\noindent{\bf Leibniz rule} From the characterization \eqref{eq:localsob2} it easily follows that $|g|\weakgrad f+|f|\weakgrad g\in L^2(\X)$ is a weak upper gradient for $fg$, so that  $fg\in\s^2(\X)$. Now assume that $f,g\geq 1$ $\mm$-a.e.. Then also $fg\geq 1$ $\mm$-a.e.\ and we can apply the chain rule with $\varphi=\log$, which is Lipschitz on the image of $f,g$ and $fg$, to get
\[
\begin{split}
\frac{\d(fg)}{fg}=\d(\log(fg))=\d(\log f+\log g)=\d\log f+\d \log g=\frac{\d f}f+\frac{\d g}g,
\end{split}
\]
which is the thesis. The general case now follows easily replacing $f,g$ by $f+C,g+C$ for $C\in\R$ large enough.
\end{proof}

\subsection{Duality and the tangent module}\label{se:tan}

\subsubsection{The module dual}

\begin{definition}[Dual of a module]
Let $\M$ be a $L^2(\X)$-normed module. Its dual $\M^*$ is the space of linear continuous maps $L:\M\to L^1(\X)$ such that
\[
L(fv)=f\, L(v),\qquad\forall f\in L^\infty(\X),\ v\in \M.
\]
We equip $\M^*$ with the operator norm, i.e. $\|L\|_*:=\sup_{v:\|v\|\leq 1} \|L(v)\|_{L^1}$. The multiplication of $f\in L^\infty(\X)$ and $L\in \M^*$ is defined as 
\[
(fL)(v):=L(fv),\qquad\forall v\in \M.
\]
Finally, the pointwise norm $|L|_*$ of $L\in\M^*$ is defined as
\[
|L|_*:=\esssup_{v: |v|\leq1\ \mm\ae}|L(v)|.
\]
\end{definition}
The only non-trivial thing to check in order to show that the structure just defined is a $L^2$-normed module is property \eqref{eq:recnorm} (which also grants that $|L|_*$ belongs to $L^2(\X)$). From the definition it is not hard to check that
\[
|L(v)|\leq |L|_*|v|\quad\mm\ae\qquad\forall v\in\M,\ L\in\M^*,
\]
and thus by integration we get $\|L(v)\|_{L^1}\leq \|v\|\||L|_*\|_{L^2}$ showing that $\|L\|_*\leq\||L|_*\|_{L^2}$.

For the opposite inequality notice  that from the basic properties of the essential supremum there is a sequence $(v_n)\subset \M$ such that $|v_n|\leq 1$ $\mm$-a.e.\ for every $n\in\N$ satisfying $|L|_*=\sup_n|L(v_n)|$. Put $\tilde v_0:=v_0$ and for $n>0$ define recursively $A_n:=\{|L(v_n)|>|L(\tilde v_{n-1})|\}$ and $\tilde v_n:=\nchi_{A_n}v_n+\nchi_{A_n^c}\tilde v_{n-1}$. Then $|\tilde v_n|\leq 1$ $\mm$-a.e.\ and the sequence $(|L(\tilde v_n)|)$ is increasing and converges $\mm$-a.e.\ to $|L|_*$. Pick $f\in L^2\cap L^\infty(\X)$ arbitrary, notice that $\|f\tilde v_n\|=\||f\tilde v_n|\|_{L^2}\leq \|f\|_{L^2}$ and thus
\[
\int|f||L(\tilde v_n)|\,\d\mm=\int |L(f\tilde v_n)|\,\d\mm\leq \|f\tilde v_n\|\,\|L\|_*=\|f\|_{L^2}\|L\|_*\qquad\forall n\in\N.
\]
By the monotone convergence theorem the integral on the left goes to $\int |f||L|_*\,\d\mm$ as $n\to\infty$, hence passing to the limit we obtain
\[
\int |f||L|_*\,\d\mm\leq\|f\|_{L^2}\|L\|_*
\]
and being this true for every $f\in L^2\cap L^\infty(\X)$ we conclude that $\||L|_*\|_{L^2}\leq\|L\|_*$, 
as desired.

\bigskip

We shall frequently use the fact that for $L:\M\to L^1(\X)$ linear and continuous we have
\begin{equation}
\label{eq:checkmstar}
L\in \M^*\quad\Leftrightarrow\quad  L(\nchi_Ev)=\nchi_EL(v)\quad\text{ for every $E\subset \X$ Borel and $v\in\M$,}
\end{equation}
which can be proved by first checking that $L(fv)=fL(v)$ holds for simple $f$ and then arguing by approximation.

\bigskip

Denote by $\M'$ the dual of $\M$ seen as Banach space, so that $\M'$ is the Banach space of linear continuous maps from $\M$ to $\R$ equipped with its canonical norm $\|\cdot\|_{'}$. Integration provides  a natural map $\Int:\M^*\to \M'$ sending $L\in \M^*$ to the operator $\Int(L)\in \M'$ defined as
\[
\Int(L)(v):=\int L(v)\,\d\mm,\qquad\forall v\in \M.
\]
\begin{proposition}\label{prop:fulldual}
The map $\Int$ is a bijective isometry, i.e.\ $\|L\|_*=\|\Int(L)\|_{'}$ for every $L\in \M^*$.
\end{proposition}
\begin{proof} The trivial bound
\[
|\Int(L)(v)|=\Big|\int L(v)\,\d\mm\Big|\leq \|L(v)\|_{L^1}\leq \|v\|\|L\|_*
\]
shows that $\|\Int(L)\|_{'}\leq \|L\|_*$. To prove the converse, fix $L\in \M^*$, $\eps>0$ and find $v\in \M$ such that $\|L(v)\|_{L^1}\geq \|v\|(\|L\|_*-\eps)$. Put $\tilde v:=\nchi_{\{L(v)\geq 0\}}v-\nchi_{\{L(v)<0\}}v$, notice that $|\tilde v|=|v|$ and $L(\tilde v)=|L(v)|$ $\mm$-a.e.\ and conclude by
\[
\|\Int(L)\|_{'}\|\tilde v\|\geq |\Int(L)(\tilde v)|=\Big|\int L(\tilde v)\,\d\mm\Big|=\|L(v)\|_{L^1}\geq \|v\|(\|L\|_*-\eps)= \|\tilde v\|(\|L\|_*-\eps)
\]
and the arbitrariness of $\eps>0$. Thus it remains to prove that $\Int$ is surjective.

Pick $\ell\in \M'$, fix $v\in \M$ and consider the map sending a Borel set $E$ to $\mu_{v}(E):=\ell(\nchi_Ev)\in\R$. It is additive and given a disjoint sequence $(E_i)$ of Borel sets we have
\[
|\mu_{v}(\cup_nE_n)-\mu_{v}(\cup_{n=1}^NE_n)|=|\mu_{v}(\cup_{n>N}E_n)|=|\ell(\nchi_{\cup_{n>N}E_n}v)|\leq\|\ell\|_{'}\|\nchi_{\cup_{n>N}E_n}v\|
\]
and since $\|\nchi_{\cup_{n>N}E_n}v\|^2=\int_{\cup_{n>N}E_n}|v|^2\,\d\mm\to 0$ by the dominate convergence theorem, we see that $\mu_{v}$ is a Borel measure. By construction, it is also absolutely continuous w.r.t.\ $\mm$ and thus it has a Radon-Nikodym derivative, which we shall denote by $L(v)\in L^1(\X)$.

The construction trivially ensures that $v\mapsto L(v)$ is linear and since for every $ E,F\subset \X$ Borel the identities  $\mu_{\nchi_Ev}(F)=\ell(\nchi_F\nchi_Ev)=\ell(\nchi_{E\cap F}v)=\mu_{v}(E\cap F)$ grant that $\int _FL(\nchi_E v)=\int_{E\cap F}L(v)$, we see that
\begin{equation}
\label{eq:spezzatoe}
L(\nchi_E v)=\nchi_EL(v)\qquad\forall v\in\M,\ E\subset\X\ \text{Borel}.
\end{equation}
Now given $v\in \M$ we put $\tilde v:=\nchi_{\{L(v)\geq 0\}}v-\nchi_{\{L(v)<0\}}v$ so that $|\tilde v|=|v|$ and, by \eqref{eq:spezzatoe} and  the linearity of $L$ we have $|L(v)|=L(\tilde v)$ $\mm$-a.e.. Then 
\[
\begin{split}
\|L(v)\|_{L^1}=\int L(\tilde v)\,\d\mm=\mu_{\tilde v,\ell}(\X)=\ell(\tilde v)\leq\|\ell\|_{'}\|\tilde v\|=\|\ell\|_{'}\|v\|,
\end{split}
\]
i.e.\ $v\mapsto L(v)$ is continuous. The conclusion follows from   \eqref{eq:spezzatoe} and \eqref{eq:checkmstar}.
\end{proof}
The Hanh-Banach theorem grants that for every $v\in\M$ there exists $\ell\in \M'$ with $\|\ell\|_{'}=\|v\|$ and $|\ell(v)|=\|v\|^2$. Putting $L:=\Int^{-1}(v)$, from the fact that the inequalities
\[
\|v\|^2=\ell(v)=\int L(v)\,\d\mm\leq\int |L|_*|v|\,\d\mm\leq \||v|\|_{L^2}\||L|_*\|_{L^2}=\|v\|\|L\|_*=\|v\|\|\ell\|_{'}=\|v\|^2
\] 
are in fact equalities we deduce  that $\mm$-a.e.\ it holds
\begin{equation}
\label{eq:perdual}
|L|_*=|v|\qquad\qquad L(v)=|v|^2.
\end{equation}
It follows that the natural embedding $\mathcal I:\M\to \M^{**}$ sending $v$ to the map $L\mapsto L(v)$, which is trivially $L^\infty$-linear, preserves the pointwise norm. Indeed, since for any $v,L$ we have $|\mathcal I(v)(L)|=|L(v)|\leq |v||L|_*$ we have $|\mathcal I(v)|_{**}\leq |v|$, while the opposite inequality comes considering $L$ such that \eqref{eq:perdual} holds. 

Modules $\M$ for which $\mathcal I$ is surjective will be called {\bf reflexive}.
\begin{proposition}[Riesz theorem for Hilbert modules and reflexivity]\label{prop:riesz} Let $\H$ be an Hilbert module and consider the map sending $v\in \H$ to $L_v\in \H^*$ given by $L_v(w):=\la v, w\ra$.

Then this map is an isomorphism of modules. In particular, Hilbert modules are reflexive.
\end{proposition}
\begin{proof} The only non-trivial claim about the map $v\mapsto L_v$ is surjectivity. To check it, let $L\in \H^*$, consider $\Int(L)\in \H'$ and apply the standard Riesz theorem to find $v\in \H$ such that
\[
\int L(w)\,\d\mm=\Int(L)(w)=\la v,w\ra_\H=\int \la v,w\ra\,\d\mm\qquad\forall w\in\H,
\]
where $\la\cdot,\cdot\ra_\H$ is the scalar product in the Hilbert space $\H$ and the last identity follows from \eqref{eq:recnorm} by polarization. Writing $\nchi_E w$ in place of $w$ in the above for $E\subset \X$ Borel arbitrary we see that $L(w)=\la v,w\ra$ $\mm$-a.e., i.e.\ $L=L_v$. The claim about reflexivity is now obvious.
\end{proof}

\begin{proposition}\label{prop:genext}
Let $\M$ be a $L^2(\X)$-normed module $V\subset \M$ a vector subspace which generates $\M$ and $L:V\to L^1(\X)$ a linear map. Assume that for some $g\in L^2(\X)$ we have
\begin{equation}
\label{eq:perext}
|L(v)|\leq g\,|v|\quad\mm\ae \qquad\forall v\in V. 
\end{equation}
Then there is a unique $\tilde L\in \M^*$ such that $\tilde L(v)=L(v)$ for every $v\in V$ and for such $\tilde L$ we have $|\tilde L|_*\leq g$.
\end{proposition}
\begin{proof}
Any extension $\tilde L$ of $L$ which is $L^\infty(\X)$-linear must be such that
\begin{equation}
\label{eq:defltilde}
\tilde L(v)=\sum_i\nchi_{E_i}L(v_i),\qquad\text{ for }v=\sum_i\nchi_{E_i}v_i
\end{equation}
where $(E_i)$ is a finite partition of $\X$ and $(v_i)\subset V$. For $\tilde L$ defined in this way, the bound \eqref{eq:perext} gives that
\[
|\tilde L(v)|=\sum_i\nchi_{E_i}|L(v_i)|\leq \sum_i\nchi_{E_i}g|v_i| =g\Big|\sum_i\nchi_{E_i}v_i\Big|=g |v|
\]
and in particular $\|\tilde L(v)\|_{L^1(\X)}\leq \|g\|_{L^2(\X)}\|v\|$. This shows that the definition \eqref{eq:defltilde} is well-posed - in the sense that $\tilde L(v)$ depends only on $v$ and not on the way to represent it as $\sum_i\nchi_{E_i}v_i$ - and that it is continuous. Since by assumption the set of $v$'s of the form $\sum_i\nchi_{E_i}v_i$ is dense in $\M$, we can uniquely extend $\tilde L$ to a continuous operator $\tilde L:\M\to L^1(\X)$. The fact that such $\tilde L$ is linear is obvious and the definition \eqref{eq:defltilde} easily gives that $\tilde L(fv)=f\tilde L(v)$ holds for simple functions $f$. Then $L^\infty$-linearity follows by approximation.
\end{proof}
We conclude with the following proposition, which in some sense says that the operations of taking the dual and of taking the $L^0$-completion (recall Theorem \ref{thm:defl0}) commute:
\begin{proposition} Let $\M$ be a $L^2$-normed module. Then the duality pairing $\M\times \M^*\to L^1(\X)$ uniquely extends to a continuous duality pairing $\M^0\times(\M^*)^0\to L^0(\X)$. Moreover, if $L:\M^0\to L^0(\X)$ is such that for some $g\in L^0(\X)$ it holds
\begin{equation}
\label{eq:bl0}
|L(v)|\leq g\, |v|\quad\mm\ae\qquad\forall v\in \M^0,
\end{equation}
then $L\in (\M^*)^0$ (in the sense of the previously defined pairing).
\end{proposition}
\begin{proof}
The claim about the unique continuous extension is a trivial consequence of the definitions. For the second part of the claim just notice that we can always find a  sequence $(E_n)$ of Borel sets such that $\nchi_{E_n}g\in L^2(\X)$ for every $n\in\N$ and $(\nchi_{E_n}g)\to g$ in $L^0(\X)$. Then from \eqref{eq:bl0} and Proposition \ref{prop:genext} above with $V=\M$ we see that the map $v\mapsto L_n(v):=\nchi_{E_n}L(v)$ belongs to $\M^*$. Since clearly  $|L_n-L_m|_*\leq|\nchi_{E_n}-\nchi_{E_m}|g$, the sequence  $(L_n)$ is Cauchy in $(\M^*)^0$ and its limit is easily seen to be equal to $L$.
\end{proof}

\subsubsection{The tangent module}

\begin{definition}[Tangent module]
The tangent module $L^2(T\X)$ is defined as the dual of the cotangent module $L^2(T^*\X)$. Its elements are called vector fields.
\end{definition}
To keep consistency with the notation used in the smooth setting, we shall denote the pointwise norm in $L^2(T\X)$ as $|\cdot|$, rather than $|\cdot|_*$, and the duality pairing between $\omega\in L^2(T^*\X)$ and $X\in L^2(T\X)$ as $\omega(X)$.
\begin{definition}[$L^2$ derivations]
A $L^2$-derivation is a linear map $L:\S^2(\X)\to L^1(\X)$ for which there is $g\in L^2(\X)$ such that
\begin{equation}
\label{eq:defder}
|L(f)|\leq g \weakgrad f\qquad\forall f\in \S^2(\X).
\end{equation}
\end{definition}
Notice that the concept of derivation has a priori nothing to do with the notion of $L^2$-normed module. It is therefore interesting to see that such notion emerges naturally from the concept of derivation, because as the following theorem shows,  derivations and vector fields are two different points of view on the same kind of object. The same result, in conjunction with the  Leibniz rule \eqref{eq:leibniz},  also shows that, although not explicitly encoded in the definition,  derivations satisfy the Leibniz rule $L(fg)=fL(g)+gL(f)$ for any $f,g\in L^\infty\cap \S^2(\X)$.
\begin{theorem}[Derivations and vector fields]\label{thm:dervf}
 For any vector field  $X\in L^2(T\X)$ the map $X\circ \d:\s^2(\X)\to L^1(\X)$ is a derivation.  
 
Conversely, given a  derivation $L$ there exists a unique vector field  $X\in L^2(T\X)$ such that  the diagram
\begin{center}
\begin{tikzpicture}[node distance=2.5cm, auto]
  \node (S) {$\s^2(\X)$};
  \node (C) [right of=S] {$L^2(T^*\X)$};
  \node (L) [below  of=C] {$L^1(\X)$};
  \draw[->] (S) to node {$\d$} (C);
  \draw[->] (C) to node {$X$} (L);
  \draw[->] (S) to node [swap] {$L$} (L);
\end{tikzpicture}
\end{center}
commutes.
\end{theorem}
\begin{proof} The first claim follows from the linearity of  $X\circ\d$, the fact that  $|X|\in L^2(\X)$ and the inequality   $|\d f(X)|\leq |X|\,|\d f|=|X|\,\weakgrad f$  valid $\mm$-a.e.\ for any $f\in \S^2(\X)$.

For the second, let $L$ be a derivation, put $V:=\{\d f:f\in\s^2(\X)\}$ and define $\tilde L:V\to L^1(\X)$ by $\tilde L(\d f):=L(f)$. Inequality \eqref{eq:defder} grants that  this is a good definition, i.e.\ $\tilde L(\d f)$ depends only on $\d f$ and not on $f$, and that
\[
| \tilde L(\d f)|\leq g |\d f|.
\]
The conclusion then follows from Proposition \ref{prop:genext} recalling that $V$ generates $L^2(T^*\X)$.
\end{proof}
Taking the adjoint of the differential leads to the notion of divergence:
\begin{definition}[Divergence]
We say that $X\in L^2(T\X)$ has divergence in $L^2$, and write $X\in  D(\div)$ provided there is $h\in L^2(\X)$ such that
\begin{equation}
\label{eq:defdiv}
\int f h\,\d\mm=-\int \d f(X)\,\d\mm\qquad\forall f\in W^{1,2}(\X).
\end{equation}
In this case we shall call $h$ the divergence of $X$ and denote it by $\div(X)$.
\end{definition}
Notice that by the density of $W^{1,2}(\X)$ in $L^2(\X)$ there is at most one $h$ satisfying \eqref{eq:defdiv}, hence the divergence is unique.

It is also easily verified that for $X\in D(\div)$ and $g\in \LIP_b(\X)$ we have $gX\in D(\div)$ with
\begin{equation}
\label{eq:leibdiv}
\div(gX)=\d g(X)+g\div (X),
\end{equation}
indeed, start observing that replacing $f$ with $\min\{\max\{f,-n\},n\}$ in \eqref{eq:defdiv} and then sending $n\to\infty$, we can reduce to check \eqref{eq:defdiv} for $f\in L^\infty\cap W^{1,2}(\X)$. For such $f$ we can apply the Leibniz rule \eqref{eq:leibniz} to get
\[
\begin{split}
\int f(\d g(X)+g\div (X))\,\d\mm=\int f \,\d g(X)-\d(fg)(X)\,\d\mm=-\int g\,\d f(X)\,\d\mm,
\end{split}
\]
which is the claim.

Notice that we are not claiming that in general $D(\div)$ contains an non-zero vector field; in this direction, see \eqref{eq:densediv}.
\subsection{Link with the metric}

\subsubsection{Pullback of a module}
The concept of pullback of a module mimics the one of pullback of a bundle. 
\begin{definition}[Maps of bounded compression]\label{def:boundcomp} Let $(\X,\mm_\X)$ and $(\Y,\mm_\Y)$ be measured spaces. We say that $\varphi:\Y\to \X$ has bounded compression provided $\varphi_*\mm_\Y\leq C \mm_\X$ for some $C>0$.
The least such constant $C$ is called \emph{compression constant} and denoted by $\comp(\varphi)$.
\end{definition}

\begin{thmdef}[Pullback module and pullback map]
Let $\M$ be a $L^2(\X)$-normed module and $\varphi:\Y\to \X$ a map of bounded compression. 

Then there exists a unique couple $(\varphi^*\M,\varphi^*)$ with $\varphi^*\M$ being a $L^2(\Y)$-normed module and $\varphi^*:\M\to \varphi^*\M$ linear and continuous such that
\begin{itemize}
\item[i)] for every $v\in \M$ it holds $|\varphi^*v|=|v|\circ\varphi$\ \ $\mm_\Y$-a.e.
\item[ii)] $\varphi^*\M$ is generated by $\{\varphi^*v\ :\ v\in \M\}$.
\end{itemize}
Uniqueness is intended up to unique isomorphism, i.e.: if $(\widetilde{\varphi^*\M},\widetilde\varphi^*)$ is another such couple, then there is a unique isomorphism $\Phi:{\varphi^*\M}\to \widetilde{\varphi^*\M}$ such that $\Phi(\varphi^*v)=\widetilde\varphi^*v$ for any $v\in\M$,
\end{thmdef}
Note: we call an element of $\varphi^*\M$ {\bf simple} if it can be written as $\sum_i \nchi_{A_i}\varphi^*v_i$ for some finite Borel partition $(A_i)$ of $\Y$ and elements $v_i\in\M$.
\begin{sketch}\\
\noindent{\bf Uniqueness} As in the proof of Theorem \ref{thm:defcot}, any such $\Phi$ must send the simple element $\sum_i \nchi_{A_i}\varphi^*v_i$ to $\sum_i \nchi_{A_i}\widetilde\varphi^*v_i$ and properties $(i),(ii)$ grant that this is a good definition and that $\Phi$ can uniquely be extended by continuity to a map which is the desired isomorphism.

\noindent{\bf Existence} Consider the set `Pre-Pullback Module' ${\sf Ppb}$ defined as
\[
{\sf Ppb}:=\{(A_i,v_i)_{i=1,\ldots,n}\ : \ n\in\N,\ (A_i)\text{ is a Borel partition of $\Y$ and }v_i\in \M\ \forall i=1,\ldots,n\},
\]
define an equivalence relation on it by declaring $(A_i,v_i)\sim (B_j,w_j)$ provided
\[
|v_i-w_j|\circ\varphi=0\qquad\mm_\Y\ae\text{ on } \ A_i\cap B_j,\qquad\forall i,j
\]
and   the map $\varphi^*:\M\to {\sf Ppb}/\sim$ which sends $v$ to the equivalence class of $(\Y,v)$. The construction now proceeds as for the  cotangent module given in Theorem \ref{thm:defcot}: one defines on ${\rm Ppb}/\sim$ a vector space structure, a multiplication by simple functions on $\Y$, a pointwise norm and a norm, then passes to the completion to conlude. We omit the details.
\end{sketch}
\begin{example}{\rm
If $\M=L^2(\X)$, then $\varphi^*\M$ is (=can be identified with) $L^2(\Y)$, the pullback map being given by $\varphi^*f=f\circ\varphi$.
}\fr\end{example}
\begin{example}\label{ex:prodpb}{\rm If $(\Y,\mm_Y)$ is the product of $(\X,\mm_\X)$ and another measured space $(\Z,\mm_\Z)$ and $\varphi:\Y\to\X$ is the natural projection, then the pullback of $\M$ via $\varphi$ is (=can be identified with) $L^2(\Z,\M)$ with the pullback map being the one assigning to  a given $v\in\M$ the function identically equal to $v$. 

Notice indeed that $L^2(\Z,\M)$ admits a canonical multiplication with functions in $L^\infty(\Y)=L^\infty(\X\times\Z )$: the product of  $z\mapsto v(z)\in \M$ and $f(x,z)\in L^\infty(\X\times \Z )$ is $z\mapsto f(\cdot,z)v(z)\in\M$. Also, on $L^2(\Z,\M)$ there is a natural pointwise norm: the one assigning to $z\mapsto v(z)\in \M$ the map $(x,z)\mapsto |v(z)|(x)$. 

The claim is now easily verified.}\fr
\end{example}
\begin{proposition}[Universal property of the pullback]\label{prop:univpb}
Let $\M$ be a $L^2(\X)$-normed module, $\varphi:\Y\to \X$ a map of bounded compression, $\NN$ a $L^2(\Y)$-normed module and $T:\M\to \NN$ linear and such that for some $C>0$ it holds
\[
|T(v)|\leq C|v|\circ\varphi\qquad \mm_\Y\ae.
\]
Then there exists a unique $L^\infty(\Y)$-linear and continuous map $\hat T:\varphi^*\M\to \NN$ such that
\[
\hat T(\varphi^*v)=T(v)\qquad\forall v\in \M.
\]
\end{proposition}
\begin{sketch} Consider the space $V:=\{\varphi^*v:v\in \M\}$, which generates $\varphi^*\M$, and the map $L:V\to \NN$ given by $L(\varphi^*v):=T(v)$, then argue as for Proposition \ref{prop:genext}.
\end{sketch}
\begin{remark}[Functoriality of the pullback]\label{rem:functpb}{\rm
A direct consequence of this last proposition is that if $\varphi:\Y\to \X$ and $\psi:\Z\to\Y$ are both of bounded compression and $\M$ is a $L^2(\X)$-normed module, then $\psi^*\varphi^*\M$ can be canonically identified to $(\psi\circ\varphi)^*\M$ via the only isomorphism which sends $\psi^*\varphi^*v$ to $(\psi\circ\varphi)^*v$ for every $v\in\M$.
}\fr\end{remark}
\begin{remark}[The case of invertible $\varphi$]\label{rem:invpb}{\rm
If $\varphi$ is invertible with inverse of bounded deformation, then the previous remark grants that $\varphi^*$ is bijective. Moreover, the right composition with $\varphi$ provides an isomorphism of $L^\infty(\X)$ and $L^\infty(\Y)$ and under this isomorphism the modules $\M$ and $\varphi^*\M$ can be identified, the isomorphism being $\varphi^*$.
}\fr\end{remark}

Consider now also the dual $\M^*$ of the module $\M$ and its pullback $\varphi^*\M^*$. There is a natural duality relation between $\varphi^*\M$ and $\varphi^*\M^*$:
\begin{proposition}
There exists a unique $L^\infty(\Y)$-bilinear and continuous map from $\varphi^*\M\times \varphi^*\M^*$ to $L^1(\Y)$ such that 
\begin{equation}
\label{eq:dualpullb}
\varphi^*\omega(\varphi^*v)=\omega(v)\circ\varphi\qquad\forall v\in \M,\ \omega\in \M^*
\end{equation}
and for such map it holds
\begin{equation}
\label{eq:bpb}
|W(V)|\leq |W|_*|V|\qquad\forall V\in\varphi^*\M,\ W\in \varphi^*\M^*.
\end{equation}
\end{proposition}
\begin{proof}
Considering simple elements $W\in \varphi^*\M^*$ and $V\in \varphi^*\M$ we see that the requirement \eqref{eq:dualpullb} and $L^\infty(Y)$-bilinearity force the definition
\begin{equation}
\label{eq:defdualpb}
W(V):=\sum_{i,j}\nchi_{A_i\cap B_j}\omega_i(v_j)\circ\varphi\qquad\text{for}\quad W=\sum_i\nchi_{A_i}\varphi^*\omega_i\quad V:=\sum_j\nchi_{B_j}\varphi^*v_j.
\end{equation}
The  bound
\[
\Big|\sum_{i,j}\nchi_{A_i\cap B_j}\omega_i(v_j)\circ\varphi\Big|\leq \sum_{i,j}\nchi_{A_i\cap B_j}|\omega_i|\circ\varphi|v_j|\circ\varphi=\sum_i\nchi_{A_i}|\omega_i|\circ\varphi\sum_j\nchi_{B_j}|v_j|\circ\varphi=|W|\,|V|
\]
shows that the above definition is well posed, in the sense that the definition of $W(V)$ depends only on $V,W$ and not on the way they are written as finite sums. The same bound also shows that \eqref{eq:bpb} holds for simple elements and that $\|W(V)\|_{L^1(Y)}\leq \|W\|_{\varphi^*M^*}\|V\|_{\varphi^*M}$. 

Since the definition \eqref{eq:defdualpb} also trivially grants that $(fW)(gV)=fgW(V)$ for $f,g$ simple, all the conclusions follow by the density of  simple elements in the respective modules.
\end{proof}
The last proposition can be read by saying that there is a natural embedding $\mathcal I$ of $\varphi^*\M^*$ into $(\varphi^*\M)^*$ which sends $W\in \varphi^*\M^*$ into the map
\[
\varphi^*\M\ni V\quad\mapsto\quad W(V)\in L^1(\Y).
\]
Routine computations shows that $\mathcal I$ is a module morphism which preserves the pointwise norm. It is natural to wonder whether it is surjective, i.e.\ whether $\varphi^*\M^*$ can be identified or not with the dual of $\varphi^*\M$. Example \ref{ex:prodpb} and Proposition \ref{prop:fulldual}  show that in general the answer is negative, because in such  case our question can be reformulated as: is the dual of $L^2(\Z,\M)$ given by $L^2(\Z,\M^*)$? It is known (see e.g.\ \cite{DiestelUhl77}) that the answer to this latter question is yes if and only if $\M^*$ has the Radon-Nikodym property and that this is ensured if $\M^*$ is separable. 

In our case we have the following result:
\begin{theorem}[Identification of $\varphi^*\M^*$ and $(\varphi^*\M)^*$]\label{thm:dualpull} Let $(\X,\sfd_\X,\mm_\X)$, $(\Y,\sfd_\Y,\mm_\Y)$ be two complete and separable metric spaces equipped with non-negative Borel measures finite on bounded sets and $\varphi:\Y\to \X$ of bounded compression. Let $\M$ be a $L^2(\X)$ normed module such that its dual $\M^*$ is separable.

Then  $\mathcal I:\varphi^*\M^*\to (\varphi^*\M)^*$ is surjective.
\end{theorem}
The proof of this result is rather technical: we shall omit it, referring to \cite{Gigli14} for the details. Here we instead prove the following much simpler statement:
\begin{proposition}
Let $(\X,\mm_\X)$ and $(\Y,\mm_\Y)$ be two measured spaces, $\varphi:\Y\to \X$ of bounded compression and $\H$ an Hilbert module on $\X$. 

Then $\mathcal I:\varphi^*\H^*\to (\varphi^*\H)^*$ is surjective.
\end{proposition}
\begin{proof} The pointwise norm of $\H$ satisfies the pointwise parallelogram identity, hence the same holds for the pointwise norm of $\varphi^*\H$ (check first the case of simple elements, then argue by approximation). Thus $\varphi^*\H$ is a Hilbert module. Now let $R:\H\to\H^*$ and $\hat R:\varphi^*\H\to(\varphi^*\H)$ be the respective Riesz isomorphisms (recall Proposition \ref{prop:riesz}), consider $\varphi^*\circ R:\H\to\varphi^*(\H^*)$ and the induced map $\widehat{\varphi^*\circ R}:\varphi^*\H\to \varphi^*(\H^*)$ as given by Proposition \ref{prop:univpb}.

It is then readily verified that $\widehat{\varphi^*\circ R}\circ \hat R^{-1}:(\varphi^*\H)^*\to\varphi^*\H^*$ is the inverse of $\mathcal I:\varphi^*\H^*\to (\varphi^*\H)^*$, thus giving the result.
\end{proof}

\subsubsection{Speed of a test plan}
With the aid of the concept of pullback of a module we can now assign to any test plan its `derivative' $\ppi'_t$ for a.e.\ $t$. The maps of bounded compression that we shall consider are the evaluation maps $\e_t$ from $C([0,1],\X)$ endowed with a test plan $\ppi$ as reference measure to $(\X,\sfd,\mm)$. In this case, we shall denote the pullback of the tangent bundle $L^2(T\X)$ via $\e_t$ by $L^2(T\X,\e_t,\ppi)$.

\begin{thmdef}\label{thm:speedplan} Let $(\X,\sfd,\mm)$ be a metric measure space such that $L^2(T\X)$ is separable and $\ppi$ a test plan.

Then for a.e.\ $t\in[0,1]$ there exists a unique vector field $\ppi'_t\in L^2(T\X,\e_t,\ppi)$ such that for every $f\in W^{1,2}(\X)$ the identity
\begin{equation}
\label{eq:vel12}
\lim_{h\to 0}\frac{f(\gamma_{t+h})-f({\gamma_t})}h=(\e_t^*\d f)(\ppi'_t)(\gamma),
\end{equation}
holds, the limit being intended in the strong topology of $L^1(\ppi)$. For these $\ppi'_t$'s we also have
\begin{equation}
\label{eq:linksp12}
|\ppi'_t|(\gamma)=|\dot\gamma_t|,\qquad \ppi\times\mathcal L^1\restr{[0,1]}\ae \ (\gamma,t).
\end{equation}
\end{thmdef}
\begin{sketch} Start observing that since $L^2(T\X)$ is separable and isometric to the Banach dual of $L^2(T^*\X)$ (Proposition \ref{prop:fulldual}), $L^2(T^*\X)$ is also separable. Then observe that since $f\mapsto (f,\d f)$ is an isometry of $W^{1,2}(\X)$ into $L^2(\X)\times L^2(T^*\X)$ with the norm $\|(f,\omega)\|^2:=\|f\|_{L^2(\X)}^2+\|\omega\|^2_{L^2(T^*\X)}$, the space $W^{1,2}(\X)$ is separable as well.

Now pick $f\in W^{1,2}(X)$, define  $[0,1]\ni t\mapsto F_t,G_t\in  L^1(\ppi)$ as
\[
F_t(\gamma):=f(\gamma_t)\qquad\qquad G_t(\gamma):=\weakgrad f(\gamma_t)|\dot\gamma_t|,
\]
and notice that \eqref{eq:localsob} can be written as
\begin{equation}
\label{eq:ptder}
|F_s-F_t|\leq \int_t^s  G_r\,\d r\qquad \ppi\ae.
\end{equation}
Integrating this bound w.r.t.\ $\ppi$ we see in particular that the map $t\mapsto F_t\in L^1(\ppi)$ is absolutely continuous. Although this is not sufficient to deduce that such curve is differentiable at a.e.\ $t$ (because the Banach space $L^1(\ppi)$ does not have the Radon-Nikodym property), the pointwise bound \eqref{eq:ptder}  grants uniform integrability of the incremental ratios $\frac{F_{t+h}-F_t}{h}$ and in turn this grants that for some $h_n\downarrow0$ the sequence $\frac{F_{\cdot+h_n}-F_\cdot}{h_n}$ converges in the weak topology of $L^1(\mathcal L^1\restr{[0,1]}\times \ppi)$ to a limit function $\Der_\cdot(f)$  which by \eqref{eq:ptder} and the definition of $G_t$ satisfies
\begin{equation}
\label{eq:normder}
|\Der_t(f)|(\gamma)\leq\weakgrad f(\gamma_t)|\dot\gamma_t|= |\e_t^*\d f|(\gamma)|\dot\gamma_t|\qquad \mathcal L^1\restr{[0,1]}\times \ppi\ae \ (t,\gamma).
\end{equation}
From the definition of $\Der_t(f)$ it also  follows that
\[
F_s-F_t=\int_t^s\Der_r(f)\,\d r\qquad\forall t,s\in[0,1],\ t<s,
\]
and this in turn implies that $\frac{F_{t+h}-F_t}{h}$ converge to $\Der_t(f)$ strongly in $L^1(\ppi)$ as $h\to 0$ for a.e.\ $t\in[0,1]$. With a little bit of work based on the fact that $W^{1,2}(\X)$ is separable, we can then see that the exceptional set of $t$'s is independent on $f$, so that for a.e.\ $t$ we have:
\[
\forall f\in W^{1,2}(\X)\quad\frac{f\circ\e_{t+h}-f\circ \e_t}{h}\text{ converge in $L^1(\ppi)$ to some }\Der_t(f) \text{ for which \eqref{eq:normder} holds.}
\]
Fix $t$ for which this holds  and let $L_t:\{\e_t^*\d f:f\in W^{1,2}(\X)\}\to L^1(\ppi)$ be defined as $L_t(\e_t^*\d f):=\Der_t(f)$. The bound \eqref{eq:normder} grants that this is a good definition, then using  Proposition \ref{prop:genext} and Theorem \ref{thm:dualpull} (recall that we assumed $L^2(T\X)$ to be separable) we deduce that there exists a unique $\ppi'_t\in L^2(T\X,\e_t,\ppi)$ such that 
\[
\e_t^*\d f(\ppi'_t)=\Der_t(f)\qquad\forall f\in W^{1,2}(\X) 
\]
and that inequality $\leq $ in \eqref{eq:linksp12} holds. To prove $\geq$ notice that for $f\in W^{1,2}\cap \LIP(\X)$ and $\gamma$ absolutely continuous the map $t\mapsto f(\gamma_t)$ is absolutely continuous. Therefore the derivative $\frac{\d}{\d t}f(\gamma_t)$ is well defined for $\ppi\times\mathcal L^1\restr{[0,1]}$-a.e.\ $(\gamma,t)$ and it is easy to check that it $\ppi\times\mathcal L^1\restr{[0,1]}$-a.e.\ coincides with $\Der_t(f)(\gamma)$. Thus $\ppi\times\mathcal L^1\restr{[0,1]}$-a.e.\  $(\gamma,t)$ we have
\[
\frac \d{\d  t}f(\gamma_t)=\e_t^*\d f(\ppi'_t)(\gamma)\leq |\e_t^*\d f|(\gamma)|\ppi'_t|(\gamma)= |\d f|(\gamma_t)|\ppi'_t|(\gamma)\leq \Lip(f)\,|\ppi'_t|(\gamma).
\]
Hence to conclude it is sufficient to show that there exists a countable family $D$ of 1-Lipschitz functions in $W^{1,2}(\X)$ such that for any absolutely continuous curve $\gamma$ we have
\begin{equation}
\label{eq:1lipcurve}
\sup_{f\in D}\frac{\d}{\d t}f(\gamma_t)\geq |\dot\gamma_t|,\qquad{\rm a.e.}\ t.
\end{equation}
Let $(x_n)\subset \X$ be countable and dense and define $f_{n,m}(x):=\max\{0,m-\sfd(x,x_n)\}$. It is clear that $f_{n,m}\in W^{1,2}\cap \LIP(\X)$ and that $\sfd(x,y)=\sup_{n,m}f_{n,m}(x)-f_{n,m}(y)$, thus for $\gamma$ absolutely continuous we have
\[
\sfd(\gamma_s,\gamma_t)=\sup_{n,m}f_{n,m}(\gamma_s)-f_{n,m}(\gamma_t)=\sup_{n,m}\int_t^s\frac{\d}{\d r}f_{n,m}(\gamma_r)\,\d r\leq \int_t^s\sup_{n,m}\frac{\d}{\d r}f_{n,m}(\gamma_r)\,\d r
\]
and the claim \eqref{eq:1lipcurve} follows.
\end{sketch}
In applications one can often find explicit expressions for the vector fields $\ppi'_t$ in terms of the data of the problem, so that this last theorem can be used to effectively calculate the derivative of $f\circ\e_t$, see for instance Remark \ref{re:speedgeo}.

\subsection{Maps of bounded deformation}
Here we introduce maps between metric measure spaces which are `first-order smooth' and see that they naturally induce a pull-back of 1-forms and, by duality, that they have a differential.
\begin{definition}[Maps of bounded deformation]
Let $(\X,\sfd_\X,\mm_\X)$ and $(\Y,\sfd_\Y,\mm_\Y)$ be metric measure spaces. A map $\varphi:\Y\to\X$ is said of bounded deformation provided it is Lipschitz and of bounded compression (Definition \ref{def:boundcomp}).
\end{definition}
A map of bounded deformation induces by left composition a map $\hat\varphi: C([0,1],\Y)\to C([0,1],\X)$. It is clear that if $\gamma$ is absolutely continuous then so is $\hat\varphi(\gamma)$ and, denoting by ${\rm ms}_t(\hat \varphi(\gamma))$ its metric speed at time $t$, that
\begin{equation}
\label{eq:lipsp}
{\rm ms}_t(\hat\varphi(\gamma))\leq \Lip(\varphi)|\dot\gamma_t|\qquad {\rm a.e.}\ t.
\end{equation}
Also, for $\mu\in\prob \Y$ such that $\mu\leq C\mm_\Y$ we  have $\varphi_*\mu\leq C\comp(\varphi)\mm_\X$. It follows that if $\ppi$ is a test plan on $\Y$, then $\hat\varphi_*\ppi$ is a test plan on $\X$. 

By duality, we now check that for $f\in\S^2(\X)$ we have $f\circ\varphi\in \S^2(\Y)$ with
\begin{equation}
\label{eq:bdp}
|\d(f\circ\varphi)|\leq \Lip(\varphi)|\d f|\circ\varphi\qquad\mm_\Y\ae.
\end{equation}
Indeed, let $\ppi$ be a test plan on $\Y$ and  notice that
\[
\begin{split}
\int|f(\varphi(\gamma_1))-f(\varphi(\gamma_0))|\,\d\ppi(\gamma)&=\int|f(\tilde\gamma_1)-f(\tilde\gamma_0)|\,\d\hat\varphi_*\ppi(\tilde\gamma)\\
\text{because $\hat\varphi_*\ppi$ is a test plan on $\X$}\qquad&\leq\iint_0^1|\d f|(\tilde\gamma_t){\rm ms}_t(\tilde\gamma)\,\d\hat\varphi_*\ppi(\tilde\gamma)\\
&=\iint_0^1|\d f|(\varphi(\gamma_t)){\rm ms}_t(\hat \varphi(\gamma))\,\d\ppi(\gamma)\\
\text{by \eqref{eq:lipsp}}\qquad&\leq\Lip(\varphi)\iint_0^1|\d f|(\varphi(\gamma_t))|\dot\gamma_t| \,\d\ppi(\gamma),
\end{split}
\]
which, by the arbitrariness of $\ppi$ and the very definition of $\S^2(\Y)$ and minimal weak upper gradient, gives the claim.

A direct consequence of this simple observation is:
\begin{thmdef}[Pull-back of 1-forms]\label{thm:pbf}
Let $\varphi:\Y\to\X$ be of bounded deformation. Then there exists a unique linear and continuous map $\varphi^*:L^2(T^*\X)\to L^2(T^*\Y)$, called pull-back of 1-forms, such that
\begin{align}
\label{eq:pullbforms1}
\varphi^*(\d f)&=\d(f\circ\varphi)&&\forall f\in \S^2(\X)\\
\label{eq:pullbforms2}
\varphi^*(g\omega)&=g\circ\varphi\,\varphi^*\omega&&\forall g\in L^\infty(\X),\ \omega\in L^2(T^*\X),
\end{align}
and for such map it holds
\begin{equation}
\label{eq:bdp2}
|\varphi^*\omega|\leq \Lip(\varphi)|\omega|\circ\varphi\qquad\mm_\Y\ae \qquad \forall\omega\in L^2(T^*\X).
\end{equation}
\end{thmdef}
\begin{proof}
For a simple form $W=\sum_i\nchi_{A_i}\d f_i\in L^2(T^*\X)$ the requirements \eqref{eq:pullbforms1},\eqref{eq:pullbforms2}  force the definition $\varphi^*W:=\sum_i\nchi_{A_i}\circ\varphi\,\d(f_i\circ\varphi)$. The inequality
\[
\big|\sum_i\nchi_{A_i}\circ\varphi\d(f_i\circ\varphi)\big|=\sum_i\nchi_{A_i}\circ\varphi|\d(f_i\circ\varphi)|\stackrel{\eqref{eq:bdp}}\leq \Lip(\varphi)\sum_i(\nchi_{A_i}|\d f_i|)\circ\varphi=\Lip(\varphi)|W|\circ\varphi
\]
shows that the definition of $\varphi^*W$ is well-posed - i.e.\ it depends only on $W$ and not on the way we write it as $\sum_i\nchi_{A_i}\d f_i$ - and that \eqref{eq:bdp2} holds for simple forms. In particular we have
\[
\|\varphi^*W\|_{L^2(T^*\Y)}\leq\Lip(\varphi)\sqrt{\int |W|^2\circ\varphi\,\d\mm_\Y }\leq\Lip(\varphi)\sqrt{\comp(\varphi)}\|W\|_{L^2(T^*\X)} ,\qquad\forall W\ \text{simple}
\]
showing that the map $\varphi^*$ so defined is continuous from the space of simple 1-forms on $\X$ to $L^2(T^*\Y)$. Hence it can be uniquely extended to a linear continuous map from $L^2(T^*\X)$ to $L^2(T^*\Y)$, which clearly satisfies \eqref{eq:bdp2}. Thus by construction we have \eqref{eq:pullbforms1} and \eqref{eq:pullbforms2} for simple functions; the validity \eqref{eq:pullbforms2} for any $g\in L^\infty(\X)$ then follows by approximation.
\end{proof}
Notice that the composition of maps of bounded deformations is of bounded deformation and by a direct verification of the characterizing properties \eqref{eq:pullbforms1}, \eqref{eq:pullbforms2} we see that
\[
(\varphi\circ\psi)^*=\psi^*\circ\varphi^*
\]
We remark that given a map of bounded deformation $\varphi:\Y\to \X$ we have two (very) different ways of considering the pull-back of 1-forms: the one defined in the previous theorem, which takes values in $L^2(T^*\Y)$, and the one in the sense of pull-back modules, which takes values in the pullback $\varphi^*L^2(T^*\X)$ of $L^2(T^*\X)$ via $\varphi$. To avoid confusion, we shall denote the latter map by $[\varphi^*]$ keeping the notation $\varphi^*$ for the former.

With this said, by duality we can now define the differential of a map of bounded deformation:
\begin{thmdef}[Differential of a map of bounded deformation]
Let $\varphi:\Y\to\X$ be of bounded deformation and assume that $L^2(T\X)$ is separable. Then there exists a unique $L^\infty(\Y)$-linear and continuous map $\d\varphi:L^2(T\Y)\to \varphi^*L^2(T\X)$, called differential of $\varphi$, such that
\begin{equation}
\label{eq:defdiffbd}
[\varphi^*\omega]\big(\d\varphi(v)\big)=\varphi^*\omega(v)\qquad\forall \omega\in L^2(T^*\X),\ v\in L^2(T\Y)
\end{equation}
and it satisfies
\begin{equation}
\label{eq:diffnorm}
|\d\varphi(v)|\leq \Lip(\varphi)|v|\quad\mm_\Y\ae \qquad\forall v\in L^2(T\Y).
\end{equation}
\end{thmdef}
\begin{proof} Let $v\in L^2(T\Y)$ and consider the map $L_v:\{\varphi^*\omega:\omega\in L^2(T\X)\}\to L^1(\Y)$ sending $\varphi^*\omega$ to $\varphi^*\omega(v)$. The bound \eqref{eq:bdp2} and the identity $|\omega|\circ\varphi=|[\varphi^*]\omega|$ give
\[
|L_v(\omega)|\leq \Lip(\varphi)|[\varphi^*]\omega||v|\quad\mm_\Y\ae \qquad\forall\omega\in L^2(T^*\X).
\]
The vector space $\{\varphi^*\omega:\omega\in L^2(T^*\X)\}$ generates $\varphi^*L^2(T^*\X)$ and  the dual of this  module is - by Theorem \ref{thm:dualpull} and the separability assumption on $L^2(T\X)$ - the module $\varphi^*L^2(T\X)$, thus by Proposition \ref{prop:genext}  we deduce that there is a unique  element in $\varphi^*L^2(T\X)$, which we will call $\d\varphi(v)$, for which \eqref{eq:defdiffbd} holds and such $\d\varphi(v)$ also satisfies \eqref{eq:diffnorm}.

It is clear that the assignment $v\mapsto\d\varphi(v)$ is $L^\infty(\Y)$-linear and since the bound \eqref{eq:diffnorm} also ensures that such assignment is continuous, the proof is completed.
\end{proof}
\begin{remark}\label{rem:diffinv}{\rm
If $\varphi$ is invertible with inverse of bounded compression, then Remark \ref{rem:invpb} tells that the pullback module $\varphi^*L^2(T\X)$ can be identified with $L^2(T\X)$ via the pullback map. Once this identification is done, the differential $\d\varphi$ can be seen as a map from $L^2(T\Y)$ to  $L^2(T\X)$ and \eqref{eq:defdiffbd} reads as
\[
\omega(\d\varphi(v))=\varphi^*\omega(v)\circ\varphi^{-1}.
\]
}\fr\end{remark}

We shall now relate the differential just built with the notion of `speed of a test plan' as given by Theorem \ref{thm:speedplan} to see that in our setting we have an analogous of the standard chain rule
\[
(\varphi\circ\gamma)'_t=\d\varphi(\gamma_t')
\]
valid in the smooth world.

As before, let $\varphi:\Y\to\X$ be of bounded deformation, denote by $\hat\varphi$ the induced map from $C([0,1],\Y)$ to $C([0,1],\X)$ and let $\ppi$ be a test plan on $\Y$. For $t\in[0,1]$ let us also denote by $\e^\X_t,\e^\Y_t$ the evaluation maps on $C([0,1],\X)$ and $C([0,1],\Y)$ respectively.

Notice that $[(\e^\Y_t)^*]\d\varphi:L^2(T\Y)\to (\e^\Y_t)^*\varphi^*L^2(T\X)$ satisfies
\[
|[(\e^\Y_t)^*]\d\varphi(v)|\leq \Lip(\varphi)|v|\circ\e^\Y_t
\]
and thus by the universal property of the pullback given in Proposition \ref{prop:univpb} we see that there is a unique $L^\infty(\ppi)$-linear and continuous map, which we shall denote by $\widehat{\d\varphi}$, from $L^2(T\Y,\e_t^\Y,\ppi)$ to $ (\e^\Y_t)^*\varphi^*L^2(T\X)$ such that
\[
\widehat{\d\varphi}([(\e^\Y_t)^*](v))=[(\e^\Y_t)^*]\d\varphi(v)\qquad\forall v\in L^2(T\Y).
\]
We observe that for such map it holds
\begin{equation}
\label{eq:ut}
\big([(\e^\Y_t)^*](\varphi^*\omega)\big)(V)=\big([(\e^\Y_t)^*][\varphi^*](\omega)\big)\big(\widehat{\d\varphi}(V)\big)\qquad\forall \omega\in L^2(T^*\X),\ V\in L^2(T\Y,\e_t^\Y,\ppi),
\end{equation}
indeed for $V$ of the form $(\e^\Y_t)^*v$ for $v\in L^2(T\Y)$ this is a direct consequence of the defining property and the conclusion for general $V$'s follows from the fact that both sides of \eqref{eq:ut} are $L^\infty(\ppi)$-linear and continuous in $V$.

With this said, we have the following result, proved in \cite{DPG16}:
\begin{proposition}[Chain rule for speeds] Assume that $L^2(T\X)$ is separable. Then for a.e.\ $t$ we have
\begin{equation}
\label{eq:diffsp}
\widehat{\d\varphi}(\ppi'_t)=[\hat\varphi^*](\hat\varphi_*\ppi)'_t\ .
\end{equation}
\end{proposition}
\begin{proof} Both sides of \eqref{eq:diffsp} define elements of $(\e^\Y_t)^*\varphi^*L^2(T\X)\sim\hat\varphi^*(\e^\X_t)^*L^2(T\X)$, where the `$\sim$' comes from the functoriality of the pull-back (Remark \ref{rem:functpb}) and $\varphi\circ\e^\Y_t=\e^\X_t\circ\hat\varphi$. Since $(\e^\Y_t)^*\varphi^*L^2(T\X)$  is the dual of $(\e^\Y_t)^*\varphi^*L^2(T^*\X)$ (by the separability assumption and Theorem \ref{thm:dualpull}), to prove \eqref{eq:diffsp} it is sufficient to test both sides against forms of the kind $[(\e^\Y_t)^*][\varphi^*](\d f)$ for $f\in \S^2(\X)$, as they generate $(\e^\Y_t)^*\varphi^*L^2(T^*\X)$ (recall Proposition \ref{prop:genext}).

Thus let $f\in \S^2(\X)$ and notice that  for a.e.\ $t$ we have
\begin{align*}
[(\e^\Y_t)^*][\varphi^*](\d f)\big(\widehat{\d\varphi}(\ppi'_t)\big)&=[(\e^\Y_t)^*](\varphi^*\d f)(\ppi'_t)&&\text{by \eqref{eq:ut}}\\
&=[(\e^\Y_t)^*](\d(f\circ\varphi))(\ppi'_t)&&\text{by \eqref{eq:pullbforms1}}\\
&=\mathop{L^1(\ppi){\rm -lim}}_{h\to 0}\frac{f\circ\varphi\circ\e^\Y_{t+h}-f\circ\varphi\circ\e^\Y_t}{h}&&\text{by definition of }\ppi'_t\\
&=\Big(\mathop{L^1(\hat\varphi_*\ppi){\rm -lim}}_{h\to 0}\frac{f\circ\e^\X_{t+h}-f\circ\e^\X_t}{h}\Big)\circ\hat\varphi&&\text{because }\varphi\circ\e^\Y_t=\e^\X_t\circ\hat\varphi\\
&=[(\e^\X_t)^*](\d f)(\hat\varphi_*\ppi)'_t\circ\hat\varphi&&\text{by definition of }(\hat\varphi_*\ppi)'_t\\
&=\big([\hat\varphi^*][(\e^\X_t)^*](\d f)\big)\big([\hat\varphi^*](\hat\varphi_*\ppi)'_t\big)&&\text{by \eqref{eq:dualpullb}}\\
&=\big([(\e^\Y_t)^*][\varphi^*](\d f)\big)\big([\hat\varphi^*](\hat\varphi_*\ppi)'_t\big)&&\text{because }\varphi\circ\e^\Y_t=\e^\X_t\circ\hat\varphi
\end{align*}
having also used Remark \ref{rem:functpb} in the last step. This is sufficient to conclude.
\end{proof}
\begin{remark}{\rm
If $\varphi$ is invertible with inverse of bounded compression we know from Remark \ref{rem:diffinv} that $\d\varphi$ can be seen as a map from $L^2(T\Y)$ to $L^2(T\X)$, thus in this case the lift of its composition with $(\e^\X_t)^*$ to $L^2(T\Y,\e^\Y_t,\ppi)$ provides a map $\widehat{\d\varphi}$ from $L^2(T\Y,\e^\Y_t,\ppi)$ to $L^2(T\X,\e^\X_t,\hat\varphi_*\ppi)$ and in this case \eqref{eq:diffsp} reads as
\[
\widehat{\d\varphi}(\ppi'_t)=(\hat\varphi_*\ppi)'_t\ .
\]
}\fr\end{remark}

\subsection{Infinitesimally Hilbertian spaces and Laplacian}

\begin{definition}[Infinitesimally Hilbertian spaces]
$(\X,\sfd,\mm)$ is said to be infinitesimally Hilbertian provided $L^2(T^*\X)$ (and thus also $L^2(T\X)$) is a Hilbert module.
\end{definition}
\begin{remark}\label{re:defih}{\rm
Since $f\mapsto (f,\d f)$ is an isometry of $W^{1,2}(\X)$ into $L^2(\X)\times L^2(T^*\X)$ endowed with the norm $\|(f,\omega)\|^2:=\|f\|_{L^2}^2+\|\omega\|_{L^2(T^*\X)}^2$, we see that if $\X$ is infinitesimally Hilbertian, then $W^{1,2}(\X)$ is a Hilbert space.

It is possible, although not entirely trivial, to show that also the converse implication holds, i.e.\ if $W^{1,2}(\X)$ is Hilbert, then so is $L^2(T^*\X)$. In fact, the original definition of infinitesimally Hilbertian spaces given in \cite{Gigli12} adopted such `$W^{1,2}$' approach, but the for the purpose of this note we preferred to start with the seemingly more powerful definition above.
}\fr\end{remark}
By Proposition \ref{prop:riesz} we know that $L^2(T^*\X)$ and $L^2(T\X)$ are isomorphic as $L^\infty$-modules. For $f\in \S^2(\X)$, the image of $\d f$ under such isomorphism is called {\bf gradient} of $f$ and denoted by $\nabla f$. Directly from \eqref{eq:chain} and \eqref{eq:leibniz} it follows that 
\begin{align*}
\nabla(\varphi\circ f)&=\varphi'\circ f\nabla f,\qquad&&\forall f\in \S^2(\X),\ \varphi\in \LIP\cap C^1(\R),\\
\nabla (fg)&=f\nabla g+g\nabla f\qquad&&\forall f,g\in L^\infty\cap \S^2(\X).
\end{align*}
\begin{remark}\label{re:hilsep}{\rm
Remark \ref{re:defih} and \eqref{eq:refsep} grant that $W^{1,2}(\X)$ is separable. Hence by Remark \ref{rem:gencot} we see that $L^2(T^*\X)$, and thus also $L^2(T\X)$, is separable. Thus all the results of the previous sections are applicable.
}\fr\end{remark}
Notice also that both  $L^2(T^*\X)$ and $L^2(T\X)$ are endowed with a pointwise scalar product.
\begin{definition}[Laplacian]\label{def:lapl2}
The space $D(\Delta)$ is the space of all functions  $f\in W^{1,2}(\X)$ such that there is $h\in L^2(\X)$ for which
\[
\int hg\,\d\mm=-\int\la\nabla f,\nabla g\ra\,\d\mm\qquad\forall g\in W^{1,2}(\X).
\]
In this case the function $h$ is called Laplacian of $f$ and denoted by $\Delta f$. 
\end{definition}
In other words, $\Delta$ is the infinitesimal generator associated to (as well as the opposite of the subdifferential of) the Dirichlet form
\begin{equation}
\label{eq:defE}
\E(f):=\left\{
\begin{array}{ll}
\displaystyle{\frac12\int|\d f|^2\,\d\mm},&\qquad\text{ if }f\in W^{1,2}(\X),\\
+\infty,&\qquad\text{ otherwise.}
\end{array}
\right.
\end{equation}
in particular is a closed operator and from the density of $\{\E<\infty\}=W^{1,2}(\X)$ in $L^2(\X)$ it follows that $D(\Delta)$ is dense in $W^{1,2}(\X)$.  It is also clear from the definitions that 
\[
f\in D(\Delta) \quad\Leftrightarrow\quad \nabla f\in D(\div)\qquad\text{ and in this case }\Delta f=\div(\nabla f),
\]
and thus recalling \eqref{eq:leibdiv} we see that 
\begin{equation}
\label{eq:densediv}
\text{on infinitesimally Hilbertian spaces the space } D(\div)\text{ is dense in }L^2(T\X).
\end{equation}
The following calculus rules are also easily established:
\begin{align}
\label{eq:chainlap}
\Delta (\varphi\circ f)=&\varphi'\circ f\Delta f+\varphi''\circ f|\nabla f|^2,&&\forall f\in \LIP_b(\X)\cap D(\Delta),\ \varphi\in C^2(\R)\\
\label{eq:leiblap}
\Delta ( f g)=&f\Delta g+g\Delta f+2\la\nabla f,\nabla g\ra&&\forall f,g\in \LIP_b(\X)\cap D(\Delta).
\end{align}
For instance, for the second notice that for $h\in W^{1,2}(\X)$ and $f,g$ as stated, we have $fh,gh\in W^{1,2}(\X)$ and thus the claim follows from
\[
\begin{split}
\int\la\nabla h,\nabla (fg)\ra\,\d\mm&=\int \la\nabla (fh),\nabla g\ra+\la\nabla(g h),\nabla f\ra-2h\la\nabla f,\nabla g\ra\,\d\mm.
\end{split}
\]
\begin{remark}{\rm
In \cite{Sauvageot89} a different construction of `$L^2$ 1-forms' has been introduced in relation to Dirichlet forms $\E$ admitting a Carr\'e du champ $\Gamma$. Adapting a bit the original presentation, the construction starts defining a symmetric bilinear map from $[L^\infty(\X)\otimes D(\E)]^2$ to $L^1(\X)$ by putting
\[
\langle f\otimes g,f'\otimes g'\rangle:=ff'\,\Gamma(g,g')\qquad \forall f,f'\in L^\infty(\X),\ g,g'\in D(\E) 
\]
and extending it by bilinearity. Then one defines the seminorm $\|\cdot\|$ on $L^\infty(\X)\otimes D(\E)$ by putting
\[
\|\omega\|^2:=\int \langle\omega,\omega\rangle\,\d\mm\qquad\forall \omega\in L^\infty(\X)\otimes D(\E),
\]
then passes to the quotient and finally to the completion. Calling $\M$ the resulting Banach space it is easy to check that it comes with the structure of a $L^2$-normed module, the pointwise norm being given by $|\omega|:=\sqrt{\langle\omega,\omega\rangle}$ and the product with $L^\infty$-functions as (the linear continuous extension of) $h\cdot(f\otimes g):=(hf)\otimes g$.

In particular, the space of forms of the kind ${\bf 1}\otimes g$,  for $g\in D(\E)$, generates $\M$ and it holds $|{\bf 1}\otimes g|=\sqrt{\Gamma(g,g)}$.

In the case of infinitesimally Hilbertian spaces, the form $\E$ defined in \eqref{eq:defE} is a Dirichlet form whose Carr\'e du champ is given (thanks to \eqref{eq:leiblap}) by $\Gamma(f,g)=\langle\nabla f,\nabla g\rangle$ and in particular $\Gamma(g,g)=|\d g|^2$. This and Theorem \ref{thm:defcot} (and Remark \ref{rem:gencot}) show that the cotangent module $L^2(T^*\X)$ and the space $\M$ coincide, meaning that the map sending $\d g$ to ${\bf 1}\otimes g$, for $g\in W^{1,2}(\X)=D(\E)$, uniquely extends to an isomorphism of modules.
}\fr\end{remark}
We conclude with a proposition (which concentrates results from \cite{AmbrosioGigliSavare11}, \cite{AmbrosioGigliSavare11-2}, \cite{Gigli12}  and \cite{Gigli13}) which is crucial in the application of this theory to the study of geometry of $\RCD$ spaces: it provides an explicit differentiation formula along (appropriate) $W_2$-geodesics. Both the statement and the proof rely on notions of optimal transport, see e.g\ \cite{Villani09}, \cite{AmbrosioGigli11}, \cite{Santambrogio15} for an introduction to the topic. Notice that the result can be read as a purely metric version of the Brenier-McCann theorem about optimal maps and $W_2$-geodesics.
\begin{theorem}[Derivation along geodesics]\label{thm:dergeo}
Let $(\X,\sfd,\mm)$ be an infinitesimally Hilbertian space and $t\mapsto\mu_t=\rho_t\mm\subset \probt\X$ a $W_2$-geodesic made os measures with uniformly bounded supports and densities. Assume also that for some, and thus any, $p\in[1,\infty)$, the map $t\mapsto \rho_t\in L^p(\mm)$ is continuous.

Then for every $f\in W^{1,2}(\X)$ the map $t\mapsto\int f\,\d\mu_t$ is $C^1([0,1])$ and the formula
\begin{equation}
\label{eq:derc1}
\frac\d{\d t}\int f\,\d\mu_t=-\int\la\nabla f,\nabla\varphi_t\ra\,\d\mu_t,\qquad\forall t\in[0,1],
\end{equation}
where $\varphi_t$ is, for every $t\in[0,1]$, Lipschitz and such that for some $s\neq t$ the function $(s-t)\varphi$ is a Kantorovich potential from $\mu_t$ to $\mu_s$.
\end{theorem}
Note: on $\RCD(K,\infty)$ spaces every $W_2$-geodesic such that $\mu_0,\mu_1$ have both bounded densities and support satisfy the assumptions (see \cite{RajalaSturm12}).
\begin{sketch}

{\bf Step 1} Let $\varphi$ be a Lipschitz Kantorovich potential from $\mu_0$ to $\mu_1$ and let $\ppi$ be a lifting of $(\mu_t)$, i.e.\ so that $(\e_t)_*\ppi=\mu_t$ for every $t\in[0,1]$, $\ppi$ is concentrated on geodesics and $(\e_0,\e_1)_*\ppi$ is an optimal plan. We claim that
\begin{equation}
\label{eq:reprgr}
\limi_{t\to 0}\int\frac{\varphi(\gamma_0)-\varphi(\gamma_t)}t\,\d\ppi(\gamma)\geq \frac12\int|\d\varphi|^2\,\d\mu_0+\frac12W_2^2(\mu_0,\mu_1).
\end{equation} 
To see this, start noticing that  $\gamma_1\in\partial^c\varphi(\gamma_0)$ for $\ppi$-a.e.\ $\gamma$ and thus for $\ppi$-a.e.\ $\gamma$ we have
\[
\varphi(z)-\varphi(\gamma_0)\leq \frac{\sfd^2(z,\gamma_1)}2-\frac{\sfd^2(\gamma_0,\gamma_1)}2\leq \sfd(z,\gamma_0)\frac{\sfd(z,\gamma_1)+\sfd(\gamma_0,\gamma_1)}2,
\]
taking the positive part, dividing by $\sfd(z,\gamma_0)$ and letting $z\to\gamma_0$ we obtain
\begin{equation}
\label{eq:daunlato}
|\d\varphi|(\gamma_0)\leq \lims_{z\to\gamma_0}\frac{(\varphi(z)-\varphi(\gamma_0))^+}{\sfd(z,\gamma_0)}\leq \sfd(\gamma_0,\gamma_1)\qquad\ppi\ae\ \gamma,
\end{equation}
where the first inequality is an easy consequence of the definition of minimal weak upper gradient and the fact that $\varphi$ is Lipschitz. On the other hand, still from $\gamma_1\in\partial^c\varphi(\gamma_0)$ for $\ppi$-a.e.\ $\gamma$ we have
\[
\varphi(\gamma_0)-\varphi(\gamma_t)\geq\frac{\sfd^2(\gamma_0,\gamma_1)}2-\frac{\sfd^2(\gamma_t,\gamma_1)}2=\sfd^2(\gamma_0,\gamma_1)(t-t^2/2) \quad\forall t\in(0,1)\qquad\ppi\ae\ \gamma.
\]
Thus
\[
\limi_{t\to 0}\int\frac{\varphi(\gamma_0)-\varphi(\gamma_t)}t\,\d\ppi(\gamma)\geq\int  \limi_{t\to 0}\frac{\varphi(\gamma_0)-\varphi(\gamma_t)}t\,\d\ppi(\gamma)\geq \int\sfd^2(\gamma_0,\gamma_1)\,\d\ppi(\gamma)
\]
and since $ \int\sfd^2(\gamma_0,\gamma_1)\,\d\ppi(\gamma)=W_2^2(\mu_0,\mu_1)$, this inequality and \eqref{eq:daunlato} give \eqref{eq:reprgr}.

{\bf Step 2} Let $\ppi$ as before, notice that it is a test plan  and let $f\in W^{1,2}(\X)$. Then
\[
\begin{split}
\int\frac{f(\gamma_t)-f(\gamma_0)}t\,\d\ppi(\gamma)&\leq\frac1t\iint_0^t|\d f|(\gamma_s)|\dot\gamma_s|\,\d s\,\d\ppi(\gamma)\\
&\leq \frac1{2t}\iint_0^t |\d f|^2\rho_s\,\d s\,\d\mm+\frac12W_2^2(\mu_0,\mu_1),
\end{split}
\]
passing to the limit noticing that $(\rho_t)\subset L^\infty$ is weakly$^*$-continuous we conclude that
\[
\lims_{t\to 0}\int\frac{f(\gamma_t)-f(\gamma_0)}t\,\d\ppi(\gamma)\leq\frac12\int|\d f|^2\,\d\mu_0+\frac12W_2^2(\mu_0,\mu_1).
\]
Write this inequality with $\eps f-\varphi$ in place of $f$ and subtract \eqref{eq:reprgr} to deduce that
\[
\lims_{t\to 0}\eps \int\frac{f(\gamma_t)-f(\gamma_0)}t\,\d\ppi(\gamma)\leq\frac12\int|\d(\eps f-\varphi)|^2-|\d\varphi|^2\,\d\mu_0.
\]
Dividing by $\eps>0$ (resp. $\eps<0$) and letting $\eps\downarrow0$ (resp. $\eps\uparrow 0$) and noticing that $\frac{|\d(\eps f-\varphi)|^2-|\d\varphi|^2}{\eps}= -2\la\nabla f,\nabla \varphi\ra+\eps|\d f|^2$ we conclude that 
\begin{equation}
\label{eq:der0}
\frac\d{\d t}\int f\,\d\mu_t\restr{t=0}=-\int\la\nabla f,\nabla\varphi\ra\,\d\mu_0.
\end{equation}
{\bf Step 3} By rescaling, we see from \eqref{eq:der0} that formula \eqref{eq:derc1} holds for any $t$, so that to conclude it remains to prove that the right hand side is continuous in $t$. Notice also that we are free in the choice of the (rescaled) Kantorovich potentials in \eqref{eq:derc1} and thus we may assume that they are equiLipschitz. Then since uniform limits of Kantorovich potentials are Kantorovich potentials, it is easy to see that to conclude it is sufficient to prove that for $t_n\to t$ and $(\varphi_{t_n})$ uniformly Lipschitz and uniformly converging to some $\varphi_t$ we have
\[
\lim_{n\to\infty}\int \la\nabla f,\nabla\varphi_{t_n}\ra\rho_{t_n}\,\d\mm=\int \la\nabla f,\nabla\varphi_{t}\ra\rho_{t}\,\d\mm.
\]
Since the $\rho_t$'s have uniformly bounded support, up to multiplying the $\varphi$'s by an appropriate cut-off we can assume that the $\varphi$'s  are bounded in $W^{1,2}(\X)$ and thus that the convergence of $(\varphi_{t_n})$ to $\varphi$ is weak in $W^{1,2}(\X)$. Thus $(\nabla\varphi_n)$ weakly converges to $\nabla\varphi_t$ in $L^2(T\X)$ and, by the assumptions on $\rho_t$, $(\rho_{t_n}\nabla f)$ strongly converges to $\rho_t\nabla f$ in $L^2(T\X)$. The thesis follows.
\end{sketch}
\begin{remark}\label{re:speedgeo}{\rm
In connection with Theorem \ref{thm:speedplan}, the proof of this last proposition can be used to show that for $\ppi$ as in the proof, the vector fields $\ppi'_t$ are defined for every $t$ (and not just for a.e.\ $t$) and are given by
\[
\ppi'_t=\e_t^*(\nabla\varphi_t).
\]
This follows noticing that for $A\subset C([0,1],\X)$ Borel with $\ppi(A)>0$, the plan $\ppi_A:=(\ppi(A))^{-1}\ppi\restr A$ is still a test plan and the curve $t\mapsto (\e_t)_*\ppi_A$ still satisfies the assumptions with the same functions $\varphi$'s.
}\fr\end{remark}
\section{Second order theory for $\RCD$ spaces}
\subsection{Definition of $\RCD$ spaces}
From now on, we shall always assume that our space satisfies the Riemannian Curvature Dimension condition  $\RCD(K,\infty)$, the definition being (\cite{AmbrosioGigliSavare12}):
\begin{definition}[$\RCD(K,\infty)$ spaces]
Let $K\in\R$. $(\X,\sfd,\mm)$ is a $\RCD(K,\infty)$ space provided:
\begin{itemize}
\item[i)] it is infinitesimally Hilbertian
\item[ii)] for some $C>0$ and $x\in \X$ it holds $\mm(B_r(x))\leq e^{Cr^2}$ for every $r>0$
\item[iii)] every $f\in W^{1,2}(\X)$ with $|\d f|\in L^\infty(\X)$ admits a Lipschitz representative $\tilde f$ with $\Lip(\tilde f)\leq \||\d f|\|_{L^\infty}$
\item[iv)] for every $f\in D(\Delta)$ with $\Delta f\in W^{1,2}(\X)$ and $g\in L^\infty(\X)\cap D(\Delta)$ with $g\geq 0$, $\Delta g\in L^\infty(\X)$, it holds the Bochner inequality:
\begin{equation}
\label{eq:boc1}
\frac12\int|\d f|^2\Delta g\,\d\mm\geq \int g\big(\la\nabla f,\nabla \Delta f\ra+K|\d f|^2\big)\,\d\mm
\end{equation}
\end{itemize}
\end{definition}
In some sense the `truly defining' properties are $(i)$ and $(iv)$, while $(ii),(iii)$ are more of a technical nature: $(ii)$ is necessary to ensure a priori that the heat flow - see below - preserves the mass, while $(iii)$ to grant that Sobolev functions determine the metric of the space (notice that there are doubling spaces supporting a Poincar\'e inequality for which $(iii)$ fails).

The {\bf heat flow} $(\h_t)$ on $\X$ is the gradient flow of (=the flow associated to) the Dirichlet form $\E$, i.e.\ for $f\in L^2(\X)$ the map $t\mapsto\h_tf\in L^2(\X)$ is the only continuous curve on $[0,\infty)$ which is absolutely continuous on $(0,\infty)$ and such that $\h_0f=f$ and
\[
\frac\d{\d t}\h_tf=\Delta\h_tf\qquad\text{a.e. }t>0.
\]
It is possible to check, we omit the details, that the heat flow satisfies the {\bf weak maximum principle}  
\[
f\leq C\quad\mm\ae\qquad\Rightarrow\qquad \h_tf\leq C\quad\mm\ae\qquad\forall t\geq0
\]
and thus it can be extended to $L^1+L^\infty(\X)$. Then from \eqref{eq:boc1} one gets the following important {\bf Bakry-\'Emery estimate}: for every  $f\in W^{1,2}(\X)$ and $t\geq 0$ it holds
\begin{equation}
\label{eq:be}
|\d \h_tf|^2\leq e^{-2Kt}\h_t(|\d f|^2)\qquad\mm\ae.
\end{equation}
Formally, this comes noticing that the derivative of $[0,t]\ni s\mapsto F(s):=\h_{t-s}(|\d \h_{s}f|^2)$ is given by
\[
\h_{t-s}\Big(-\Delta (|\d \h_{s}f|^2)+2\la\nabla \h_{s}f,\nabla\Delta h_{s}f\ra \Big)
\]
and this is $\leq -2KF(s)$ by the Bochner inequality \eqref{eq:boc1}. Then one concludes with the Gronwall's Lemma.

We shall also make use of the $L^\infty-\Lip$ regularization: for $f\in L^\infty(\X)$ and $t>0$ we have $\h_tf\in \LIP(\X)$ with
\begin{equation}
\label{eq:lilip}
\sqrt{2\int_0^te^{2Ks}\,\d s}\ \Lip(\h_tf)\leq \|f\|_{L^\infty}.
\end{equation}
This, again formally, follows integrating in $s\in[0,t]$ the bound
\[
\frac\d{\d s}\h_s(|\h_{t-s}f|^2)=\h_s\Big(\Delta|\h_{t-s}f|^2-2\h_{t-s}f\Delta\h_{t-s}f\Big)\stackrel{\eqref{eq:leiblap}}=2\h_s(|\d\h_{t-s}f|^2)\stackrel{\eqref{eq:be}}\geq 2e^{2Ks}|\d\h_tf|^2,
\]
then using the weak maximum principle and  Property $(iii)$ in the definition of $\RCD$ spaces.

\subsection{Measure-valued Laplacian and test functions}
A key tool that we shall use to develop second order calculus on $\RCD$ spaces is the notion of `test function' introduced in \cite{Savare13}:
\[
\test\X:=\Big\{f\text{ bounded, Lipschitz, in $D(\Delta)$ with $\Delta f\in  W^{1,2}(\X)$}\Big\}.
\]
From \eqref{eq:lilip} and general regularization properties of the heat flow we have that
\[
f\in L^2\cap L^\infty(\X),\ f\geq 0\qquad\Rightarrow\qquad \h_tf\in\test\X,\ \h_tf\geq 0\qquad\forall t>0
\] 
and thus in particular that  $\test\X$ is dense in $W^{1,2}(\X)$. To analyze the properties of test functions it is useful to introduce the following notion, coming from \cite{Gigli12}:
\begin{definition}[Measure-valued Laplacian]
Let $f\in W^{1,2}(\X)$. We say that $f$ has a measure-valued Laplacian, and write $f\in D(\bd)$, provided there exists a  Borel measure $\mu$ on $\X$ finite on bounded sets such that
\[
\int g\,\d\mu=-\int\la\nabla f,\nabla g\ra\,\d\mm\qquad\text{for every $g\in\LIP(\X)$  with bounded support.}
\]
In this case the measure $\mu$, which is clearly unique, will be denoted by $\bd f$.
\end{definition}
It is readily verified that this concept is fully compatible with the one given in Definition \ref{def:lapl2}, in the sense that
\[
f\in D(\Delta)\quad\Leftrightarrow\quad f\in D(\bd)\ \text{with }\bd f\ll\mm\text{ and $\frac{\d\bd f}{\d\mm}\in L^2(\X)$, and in this case }\bd f=\Delta f\,\mm,
\]
and one can check that
\begin{equation}
\label{eq:deltal1}
f\in D(\bd),\quad |\d f|\in  L^1(\X)\qquad\Rightarrow\qquad \bd f(\X)=0
\end{equation}
(this is trivial if $\mm(\X)<\infty$, for the general case one approximates the constant 1 with functions with uniformly bounded Laplacian).

We then have the following crucial property, proved in \cite{Savare13}, which is the first crucial step towards second-order calculus in $\RCD$ spaces: among others, it provides Sobolev regularity for $|\d f|^2$ for any $f\in\test\X$ (in contrast, without any lower Ricci bound it seems impossible to exhibit non-constant functions $f$ for which $|\d f|$ has any kind of regularity).
\begin{theorem}\label{thm:savare} Let $f\in \test\X$. Then $|\d f|^2\in D(\bd)\subset W^{1,2}(\X)$ and
\begin{equation}
\label{eq:boc2}
\frac12\bd|\d f|^2\geq \big(\la\nabla f,\nabla \Delta f\ra+K|\d f|^2\big)\mm.
\end{equation}
\end{theorem}
\begin{sketch} From the fact that $|\d f|^2,\la\nabla f,\nabla \Delta f\ra+K|\d f|^2\in L^2(\X)$ one can check that \eqref{eq:boc1} holds for any $g\in D(\Delta)$ non-negative. Picking $g:=\h_t(|\d f|^2)$ we obtain
\[
\begin{split}
\int |\d \h_{t/2}(|\d f|^2)|^2\,\d\mm=-\int |\d f|^2\Delta \h_t(|\d f|^2)\,\d\mm&\stackrel{\eqref{eq:boc1}}\leq -\int \h_t(|\d f|^2)\big(\la\nabla f,\nabla \Delta f\ra+K|\d f|^2\big)\mm\\
&\leq \||\d f|^2\|_{L^\infty}\int\big| \la\nabla f,\nabla \Delta f\ra+K|\d f|^2\big|\mm,
\end{split}
\]
so that letting $t\downarrow0$ we conclude that $|\d f|^2\in W^{1,2}(\X)$. Now, at least if $\X$ is  compact, $|\d f|^2\in D(\bd)$ and \eqref{eq:boc2} both follow noticing that from \eqref{eq:boc1} we have that the linear operator 
\[
C(\X)\cap D(\Delta )\ni g\mapsto L(g):= \int \Delta g\frac{|\d f|^2}2- g\big(\la\nabla f,\nabla \Delta f\ra+K|\d f|^2\big)\,\d\mm
\]
is such that $L(g)\geq 0$ for $g\geq 0$. Hence it must coincide with the integral of $g$ w.r.t.\ a non-negative measure.
\end{sketch}
A direct, and important, property that follows from the above is that 
\[
\text{$\test\X$ is an algebra. }
\]Indeed, in checking that $fg\in\test\X$ for $f,g\in\test\X$ the only non-trivial thing to prove is that $\Delta(fg)\in W^{1,2}(\X)$. Since it is clear that $f\Delta g,g\Delta f\in W^{1,2}(\X)$,  by the Leibniz rule for the Laplacian \eqref{eq:leiblap} to conclude it is sufficient to show that $\la\nabla f,\nabla g\ra\in W^{1,2}(\X)$. This follows by polarization from Theorem \ref{thm:savare}.
\subsection{The space $W^{2,2}(\X)$}

\subsubsection{Tensor product of Hilbert modules}
Let $\H_1,\H_2$ be two Hilbert modules on $\X$ and denote by $\H_1\otimes_{\rm Alg}\H_2$ their tensor products as $L^\infty$-modules, so that  $\H_1\otimes_{\rm Alg}\H_2$ can be seen as the space of formal finite sums of objects of the kind $v_1\otimes v_2$ with $(v_1,v_2)\mapsto v_1\otimes v_2$ being $L^\infty$-bilinear. 

We define the $L^\infty$-bilinear and symmetric map $:$ from $[\H_1\otimes_{\rm Alg}\H_2]^2$ to $L^0(\X)$ by putting
\[
(v_1\otimes v_2):(v'_1\otimes v'_2):=\la v_1,v_1'\ra_1\la v_2,v_2'\ra_2
\]
where $\la\cdot,\cdot\ra_i$ is the pointwise scalar product on $\H_i$, $i=1,2$, and extending it by $L^\infty$-bilinearity. It is readily verified that this definition is well posed and that the resulting map is positively definite in the sense that for any $A\in \H_1\otimes_{\rm Alg}\H_2$ and $E\subset \X$ Borel it holds
\[
\begin{split}
A:A&\geq 0\quad\mm\ae\\
A:A&=0\quad\mm\ae\text{ on}\ E\qquad\text{ if and only if }\qquad A=0\quad\mm\ae\text{ on}\  E.
\end{split}
\]
Then define the \emph{Hilbert-Schimdt} pointwise norm as
\[
|A|_\HS:=\sqrt{A:A}\in L^0(\X)
\]
and the tensor product norm as
\[
\|A\|_{\H_1\otimes\H_2}:=\sqrt{\int |A|_\HS^2\,\d\mm}\in[0,+\infty].
\]
We are now ready to give the following definition:
\begin{definition}[Tensor product of Hilbert modules]
The space $\H_1\otimes\H_2$ is defined as the completion of 
\[
\Big\{A\in\H_1\otimes_{\rm Alg}\H_2\ :\ \|A\|_{\H_1\otimes\H_2}<\infty\Big\}
\]
w.r.t.\ the tensor product norm $\|\cdot\|_{\H_1\otimes\H_2}$.
\end{definition}
The multiplication by $L^\infty$ functions in $\H_1\otimes_{\rm Alg}\H_2$ is easily seen to induce by continuity a multiplication by $L^\infty$-functions on $\H_1\otimes\H_2$ which together with the pointwise norm $|\cdot|_\HS$ show that $\H_1\otimes\H_2$ comes with  the structure of $L^2$-normed module. Moreover, since $|\cdot|_\HS$ satisfies the pointwise parallelogram identity, $\H_1\otimes\H_2$ is in fact a Hilbert module.

\bigskip

If $\H_1=\H_2$, the tensor product will be denoted $\H^{\otimes 2}$. In this case the map $v_1\otimes v_2\mapsto v_2\otimes v_1$ on $\H_1\otimes_{\rm Alg}\H_2$ induces an automorphism $A\mapsto A^t$, called transposition, on $\H^{\otimes 2}$ and for a generic $A\in \H^{\otimes 2}$ we put
\[
A_{\sf Sym}:=\frac{A+A^t}2\qquad\qquad A_{\sf Asym}:=\frac{A-A^t}2
\]
for the symmetric and antisymmetric parts of $A$, respectively. It is then clear that
\begin{equation}
\label{eq:normasym}
|A|_\HS^2=|A_{\sf Sym}|_\HS^2+|A_{\sf Asym}|_\HS^2\quad\mm\ae\qquad\forall A\in \H^{\otimes 2}.
\end{equation}
We shall write $L^2((T^*)^{\otimes 2}\X)$ (resp. $L^2(T^{\otimes 2}\X)$) for the tensor product of $L^2(T^*\X)$ (resp. $L^2(T\X)$) with itself. These modules are one the dual of the other and we shall typically write $A(X,Y)$ in place of $A(X\otimes Y)$ for  $A\in L^2((T^*)^{\otimes 2}\X)$ and $X\otimes Y\in L^2(T^{\otimes 2}\X)$.

 Notice that being $L^2(T^*\X)$ separable (Remark \ref{re:hilsep}), so is  $L^2((T^*)^{\otimes 2}\X)$. Same for $L^2(T^{\otimes 2}\X)$.

\subsubsection{Definition of $W^{2,2}(\X)$}

Recall that on a smooth Riemannian manifold, the Hessian of the smooth function $f$ is characterized by the validity of the identity
\[
2\He (f)(\nabla g_1,\nabla g_2)=\la\nabla(\la\nabla f,\nabla g_1\ra),\nabla g_2\ra+\la\nabla(\la\nabla f,\nabla g_2\ra),\nabla g_1\ra-\la\nabla f,\nabla(\la\nabla g_1,\nabla g_2\ra)\ra
\]
for any smooth functions $g_1,g_2$. This motivates the following definition:
\begin{definition}[The space $W^{2,2}(\X)$ and the Hessian]
The space $W^{2,2}(\X)$ is the set of all the functions $f\in W^{1,2}(\X)$ for which there exists $A\in L^2((T^*)^{\otimes 2}\X)$ such that
\begin{equation}
\label{eq:defhess}
\begin{split}
2\int hA(\nabla g_1,\nabla g_2)\,\d\mm=-\int \la\nabla f,\nabla g_1\ra\div(h\nabla g_2)&+\la\nabla f,\nabla g_2\ra\div(h\nabla g_1)\\
&\qquad+h\la\nabla f,\nabla\la\nabla g_1,\nabla g_2\ra\ra\,\d\mm
\end{split}
\end{equation}
for every $g_1,g_2\in\test \X$ and $h\in\LIP_b(\X)$. Such $A$ will be called Hessian of $f$ and denoted by $\He f$. The space $W^{2,2}(\X)$ is equipped with the norm
\[
\|f\|_{W^{2,2}(\X)}^2:=\|f\|_{L^2(\X)}^2+\|\d f\|_{L^2(T^*\X)}^2+\|\He f\|_{L^2((T^*)^{\otimes 2}\X)}^2.
\]
\end{definition}
From the density of $\test\X$ in $W^{1,2}(\X)$ is easily follows that the Hessian, if it exists, is unique and thus in particular the $W^{2,2}$-norm is well defined. Notice that in giving the above definition we used in a crucial way Theorem \ref{thm:savare} to grant that $\la\nabla g_1,\nabla g_2\ra\in W^{1,2}(\X)$ so that  the last addend in the integral in \eqref{eq:defhess} is well defined.

The following is  easily verified:
\begin{theorem}\label{thm:basew2} We have:
\begin{itemize}
\item[i)] $W^{2,2}(\X)$ is a separable Hilbert space.
\item[ii)] The Hessian is a closed operator, i.e. the set $\{(f,\He(f)) : f\in W^{2,2}(\X)\}$ is a closed subset of $W^{1,2}(\X)\times L^2((T^*)^{\otimes 2}\X)$
\item[iii)] For every $f\in W^{2,2}(\X)$ the Hessian $\He(f)$ is symmetric, i.e.\ $\He(f)^t=\He (f)$.
\end{itemize}
\end{theorem}
\begin{proof}
For given $g_1,g_2,h\in\test \X$ the left (resp. right) hand side of \eqref{eq:defhess} is continuous w.r.t.\ $A\in L^2((T^*)^{\otimes 2}\X)$ (resp.\ $f\in W^{1,2}(\X)$). Point $(ii)$ and the completeness of $W^{2,2}$ follow. The fact that the $W^{2,2}$-norm satisfies the parallelogram rule is obvious. For the separability, notice that $L^2(\X)\times L^2(T^*\X)\times L^2((T^*)^{\otimes 2}\X)$ endowed with its natural Hilbert structure is separable and that the map
\[
W^{2,2}(\X)\ni f\quad\mapsto\quad (f,\d f,\He f)\in L^2(\X)\times L^2(T^*\X)\times L^2((T^*)^{\otimes 2}\X)
\]
is an isometry. Point $(iii)$ comes from the symmetry in $g_1,g_2$ of \eqref{eq:defhess}.
\end{proof}
\begin{remark}{\rm
As the example of weighted Riemannian manifold shows, in general the Laplacian is not the trace of the Hessian.
}\fr\end{remark}

\subsubsection{Existence of $W^{2,2}$ functions}
It is not at all obvious that $W^{2,2}(\X)$ contains any non-constant function. This (and much more) is ensured by the following crucial Lemma which is about the self-improving of Bochner inequality. Read in the smooth setting, the claim says that for the vector field $X:=\sum_ig_i\nabla f_i$ and the 2-tensor $A:=\sum_j\nabla h_j\otimes\nabla h_j$ it holds
\begin{equation}
\label{eq:dariscr}
\left|\nabla X:A\right|^2\leq  \Big( \Delta\frac{|X|^2}{2}+ \la X, \Delta_\Ho X\ra-K|X|^2-|(\nabla X)_{\sf Asym}|_\HS^2\Big)\ \left|A\right|^2_\HS,
\end{equation}
see also Lemma \ref{le:riscritto}. Given that for the moment we don't have the covariant derivative and the Hodge Laplacian, we have to state \eqref{eq:dariscr} by `unwrapping' these operators.

From now on, we shall denote by $\mes(\X)$ the space of finite Borel measures on $\X$ equipped with the total variation norm. Then  for $f,g,h\in\test\X$ it will be useful to introduce $\Ggamma_2(f,g)\in\mes(\X)$ and $H[f](g,h)\in L^1(\X)$ as
\[
\begin{split}
\Ggamma_2(f,g)&:=\frac12\Big(\bd(\la\nabla f,\nabla g\ra)-\big(\la\nabla f,\nabla\Delta g\ra+\la\nabla g,\nabla\Delta f\ra\big)\mm\Big)\\
H[f](g,h)&:=\frac12\Big(\langle\nabla(\la\nabla f,\nabla g\ra),\nabla h\rangle+\la\nabla(\la\nabla f,\nabla h\ra),\nabla g\ra-\langle\nabla f,\nabla(\la\nabla g,\nabla h\ra)\rangle\Big)
\end{split}
\]
We shall also write 
\[
\Ggamma_2(f,g)=\ggamma_2(f,g)\mm+\Ggamma^s_2(f,g),\qquad\text{ with }\qquad\Ggamma_2^s(f,g)\perp\mm.
\]
We then have the following:
\begin{lemma}[Key inequality]\label{le:lemmachiave}
Let $n,m\in\N$ and $f_i,g_i,h_j\in\test\X$, $i=1,\ldots,n$, $j=1,\ldots,m$. Define the measure $\mu=\mu\big((f_i),(g_i)\big)\in \mes (\X)$ as
\[
\begin{split}
\mu\big((f_i),(g_i)\big):=&\sum_{i, i'} g_i g_{ i'}\big(\Ggamma_2(f_i,f_{ i'})-K\la\nabla f_i,\nabla f_{i'}\ra\mm\big)\\
&\qquad+\Big(2g_iH[f_i](f_{i'},g_{i'})+\frac{\la\nabla f_i,\nabla f_{i'}\ra\la\nabla g_i,\nabla g_{i'}\ra+\la\nabla f_i,\nabla g_{i'}\ra\la\nabla g_i,\nabla f_{i'}\ra}2\Big)\mm
\end{split}
\]
and write it as  $\mu=\rho \mm+\mu^s$ with $\mu^s\perp\mm$.

Then 
\begin{equation}
\label{eq:partesing1}
\mu^s\geq 0
\end{equation}
and
\begin{equation}
\label{eq:parteac1}
\bigg|\sum_{i,j}\la\nabla f_i,\nabla h_j\ra\la\nabla g_i,\nabla h_j\ra+g_iH[f_i](h_j, h_j)\bigg|^2\leq \rho \sum_{j, j'}|\la\nabla h_j,\nabla h_{j'}\ra|^2.
\end{equation}
\end{lemma}
\begin{sketch} We shall prove the thesis in the simplified case $n=m=1$ and $g_1\equiv 1$ (this is the original argument in \cite{Bakry83} as adapted to $\RCD(K,\infty)$ spaces in \cite{Savare13}): in this case the measure $\mu$ is given by $\mu=\Ggamma_2(f,f)-K\la\nabla f,\nabla f\ra\mm$. Then \eqref{eq:partesing1} follows from \eqref{eq:boc2} and \eqref{eq:parteac1} reads as
\begin{equation}
\label{eq:parteac2}
\big| H[f](h,h)\big|^2\leq \Big(\ggamma_2(f,f)-K\langle\nabla f,\nabla f\rangle\Big)|\nabla h|^4.
\end{equation}
For $\lambda,c\in\R$ define $\Phi_{\lambda,c}=\Phi_{\lambda,c}(f,h):=\lambda f +h^2-2ch \in \test\X$. It is only a matter of computations to check that
\[
\ggamma_2(\Phi_{\lambda,c},\Phi_{\lambda,c})- K|\nabla\Phi_{\lambda,c}|^2=\lambda^2\big(\ggamma_2(f,f)-K|\nabla f|^2\big)+4\lambda H[f](h,h)+4|\nabla h|^4+(h-c)F_{\lambda,c}
\]
for some $F_{\lambda,c}\in L^1(\X,\mm)$ so that   $c\mapsto F_{\lambda,c}\in L^1(\X,\mm)$ is continuous. It follows that $\mm$-a.e.\ the inequality $\ggamma_2(\Phi_{\lambda,c},\Phi_{\lambda,c})- K|\nabla\Phi_{\lambda,c}|^2\geq 0$ (which comes from \eqref{eq:boc2}) holds for any $c\in\R$.  Hence for $\mm$-a.e.\ $x$ we can take $c=h(x)$ and conclude that
\[
\lambda^2\big(\ggamma_2(f,f)-K|\nabla f|^2\big)+4\lambda H[f](h,h)+4|\nabla h|^4\geq 0\qquad\mm\ae
\]
and \eqref{eq:parteac2} follows by the arbitrariness of $\lambda\in\R$. 

The general case follows by a similar optimization argument using $\Phi(f_i,g_i,h_j)$ in place of $\Phi(f,h)$ for $\Phi$ given by
\[
\Phi(x_1,\ldots,x_n,y_1,\ldots,y_n,z_1,\ldots,z_m):=\sum_i(\lambda x_iy_i+a_ix_i-b_iy_i)+\sum_jz^2_j-2c_jz_j\ ,
\]
we omit the details.
\end{sketch}
The first important consequence of this lemma is the following result, which shows in particular that $W^{2,2}(\X)$ is dense in $W^{1,2}(\X)$.
\begin{theorem}\label{thm:key2}
Let $f\in\test\X$. Then $f\in W^{2,2}(\X)$ and
\begin{equation}
\label{eq:hess34}
|\He f|^2_\HS\leq \ggamma_2(f,f)-K|\nabla f|^2,\qquad\mm\ae,
\end{equation}
and moreover for every $g_1,g_2\in\test\X$ it holds
\begin{equation}
\label{eq:hess12}
H[f](g_1,g_2)=\He f(\nabla g_1,\nabla g_2),\qquad\mm\ae.
\end{equation}
\end{theorem}
\begin{proof} We apply Lemma \eqref{le:lemmachiave} with $n=1$ for given functions $f,h_j\in\test\X$, $j=1,\ldots,m$ and $g\equiv 1$ (this is admissible at least if $\mm(\X)<\infty$, in the general case an approximation argument is required). In this case inequality \eqref{eq:parteac1}   reads, also recalling  the definition of pointwise norm on $L^2(T^{\otimes 2}\X)$, as:
\begin{equation}
\label{eq:perA}
\bigg|\sum_{j}H[f](h_j, h_j)\bigg|\leq \sqrt{\ggamma_2(f,f)-K|\nabla f|^2}\,\Big|\sum_{j}\nabla h_j\otimes\nabla h_j\Big|_\HS,\qquad\mm\ae.
\end{equation}
Now notice that  for arbitrary $h_j,h'_j\in\test\X $, $g_j\in\LIP_b(\X)$ we  have
\[
\begin{split}
g_j H[f](h_j, h'_j)&=\frac12g_j\Big(H[f](h_j+ h'_j,h_j+ h'_j)- H[f](h_j, h_j)-H[f](h'_j, h'_j)\Big)\\
g_j\frac{\nabla h_j\otimes\nabla h_j'+\nabla h_j'\otimes\nabla h_j}2&=g_j\frac{\nabla (h_j+h_j')\otimes\nabla(h_j+ h_j')-\nabla h_j\otimes\nabla h_j-\nabla h_j'\otimes\nabla h_j'}2,
\end{split}
\]
hence taking into account the trivial inequality $|A_{\sf Sym}|_\HS\leq|A|_\HS$ $\mm$-a.e.\ (recall \eqref{eq:normasym}) for $A:=\sum_jg_j\nabla h_j\otimes\nabla h_j'$,
from  \eqref{eq:perA} we obtain
\begin{equation}
\label{eq:perA2}
\begin{split}
\Big|\sum_jg_jH[f](h_j, h'_j)\Big|&\leq \sqrt{\ggamma_2(f,f)-K|\nabla f|^2}\,\bigg|\sum_jg_j\frac{\nabla h_j\otimes\nabla h_j'+\nabla h_j'\otimes\nabla h_j}2\bigg|_\HS\\
&\leq \sqrt{\ggamma_2(f,f)-K|\nabla f|^2}\Big|\sum_jg_j\nabla h_j\otimes\nabla h_j'\Big|_\HS.
\end{split}
\end{equation}
Now let  $V\subset L^2(T^{\otimes 2}\X)$ be the space of   linear combinations of tensors of the form $g\nabla h\otimes\nabla h'$ for $h,h'\in\test\X$, $g\in\LIP_b(\X)$  and define   $A:V\to  L^0(\X)$ as
\[
A\Big(\sum_{j}g_j\nabla h_j\otimes\nabla h'_j\Big):=\sum_jg_jH[f](h_j, h'_j).
\]
From \eqref{eq:perA2} we see that this is a good definition, i.e.\ that $A(T)$ depends only on $T$. Moreover, recalling that by \eqref{eq:partesing1} we have $\Ggamma_2^s(f,f)\geq 0$, we obtain
\begin{equation}
\label{eq:intgamma2}
\int\ggamma_2(f,f)-K|\nabla f|^2\,\d\mm\leq \Ggamma_2(f,f)(\X)-K\int |\nabla f|^2\,\d\mm\stackrel{\eqref{eq:deltal1}}=\int (\Delta f)^2-K|\nabla f|^2\,\d\mm
\end{equation}
hence from  \eqref{eq:perA2} we deduce that
 \[
 \|A(T)\|_{L^1(\X)}\leq \sqrt{\int (\Delta f)^2-K|\nabla f|^2\,\d\mm}\ \|T\|_{L^2(T^{\otimes 2}\X)},\qquad\forall T\in V.
 \]
It is readily verified that $V$ is dense in $L^2(T^{\otimes 2}\X)$, therefore $A$ can be uniquely extended to a continuous linear operator from $L^2(T^{\otimes 2}\X)$ to $L^1(\X)$ which is readily checked to be $L^\infty$-linear. In other words,  $A\in L^2((T^*)^{\otimes 2}\X)$.

Now let   $h_1,h_2\in\test\X$, $g\in \LIP_b(\X)$ be arbitrary and notice that we have
\[
\int A(g\nabla h_1\otimes \nabla h_2)\,\d\mm=2\int gH[f](h_1,h_2)\,\d\mm
\]
and, by the definition  of $H[f]$ and after an integration by parts, that
\[
\begin{split}
&2\int gH[f](h_1,h_2)\,\d\mm\\
&=\int -\la\nabla f,\nabla h_1\ra\div (g\nabla h_2)-\la\nabla f,\nabla h_2\ra\div (g\nabla h_1)-g\big<\nabla f,\nabla\la\nabla h_1,\nabla h_2\ra\big>\,\d\mm.
\end{split}
\]
These show that  $f\in W^{2,2}(\X)$ with $\He f=A$ and  that \eqref{eq:hess12} holds. For \eqref{eq:hess34} notice that \eqref{eq:perA2} can be restated as
\[
|\He f(T)|\leq  \sqrt{\ggamma_2(f,f)-K|\nabla f|^2} \,\,|T|_\HS,\qquad\forall T\in V,
\]
and use once again the density of $V$ in $L^2(T^{\otimes 2}\X)$  to conclude.
\end{proof}
In particular, we have the following important corollary:
\begin{corollary}\label{cor:bello} We have $D(\Delta)\subset   W^{2,2}(\X)$ and
\begin{equation}
\label{eq:key}
\int |\He f|^2_\HS\,\d\mm \leq \int(\Delta f)^2-K|\nabla f|^2\,\d\mm,\qquad\forall f\in D(\Delta).
\end{equation}
\end{corollary}
\begin{sketch} For $f\in\test\X$ the claim follows   integrating \eqref{eq:hess34} and recalling \eqref{eq:intgamma2}. The general case is then achieved by approximation recalling that the Hessian is a closed operator.
\end{sketch}
Such corollary ensures that the following definition is meaningful:
\begin{definition}
We define $H^{2,2}(\X)$ as the $W^{2,2}$-closure of $D(\Delta)\subset W^{2,2}(\X)$.
\end{definition}
It is not hard to check that $H^{2,2}(\X)$ also coincides with the $W^{2,2}(\X)$ closure of $\test\X$; on the other hand it is important to underline that it is not at all clear whether $H^{2,2}(\X)$ coincides with $W^{2,2}(\X)$ or not.

\subsubsection{Calculus rules}

\begin{proposition}[Product rule for functions]\label{prop:prodfunct}
Let $f_1,f_2\in \LIP_b\cap W^{2,2}(\X)$. Then $f_1f_2\in W^{2,2}(\X)$ and the formula
\begin{equation}
\label{eq:leibhess}
\He({f_1f_2})=f_2\He {f_1}+f_1\He {f_2}+\d f_1\otimes\d f_2+\d f_2\otimes \d f_1,\qquad\mm\ae
\end{equation}
holds.
\end{proposition}
\begin{proof} It is obvious that $f_1f_2\in W^{1,2}(\X)$ and that  the right hand side of \eqref{eq:leibhess} defines an object in $L^2((T^*)^{\otimes 2}\X)$. Now  let $g_1,g_2\in\test \X$, $h\in \LIP_b(\X)$ be arbitrary and notice that
\[
\begin{split}
-\la\nabla(f_1f_2),\nabla g_1\ra\,\div (h\nabla g_2)&=-f_1\la\nabla f_2,\nabla g_1\ra\,\div (h\nabla g_2)-f_2\la\nabla f_1,\nabla g_1\ra \,\div (h\nabla g_2)\\
&=-\la \nabla f_2,\nabla g_1\ra\,\div(f_1h\nabla g_2)+h\la\nabla f_2,\nabla g_1\ra\la\nabla f_1,\nabla g_2\ra\\
&\qquad-\la\nabla f_1,\nabla g_1\ra\,\div(f_2h\nabla g_2)+h\la\nabla f_1,\nabla g_1\ra\la\nabla f_2,\nabla g_2\ra.
\end{split}
\]
Exchanging the roles of $g_1,g_2$, noticing that
\[
-h\big<\nabla(f_1f_2),\nabla\la \nabla g_1,\nabla g_2\ra\big>=-hf_1\big<\nabla f_2,\nabla\la\nabla g_1,\nabla g_2\ra\big>-hf_2\big<\nabla f_1,\nabla\la\nabla g_1,\nabla g_2\ra\big>,
\]
adding everything up, integrating and observing that $f_1h,f_2h\in \LIP_b(\X)$ we conclude.
\end{proof}

\begin{proposition}[Chain rule]\label{prop:chainhess}
Let $f\in \LIP\cap W^{2,2}(\X)$ and $\varphi:\R\to\R$ a $C^{2}$ function with uniformly bounded first and second derivative (and $\varphi(0)=0$ if $\mm(\X)=+\infty$). 

Then $\varphi\circ f\in W^{2,2}(\X)$ and the formula
\begin{equation}
\label{eq:chainhess}
\He({\varphi\circ f})=\varphi''\circ f\, \d f\otimes\d f+\varphi'\circ f\,\He f,\qquad\mm\ae
\end{equation}
holds.
\end{proposition}
\begin{proof} It is obvious that $\varphi\circ f\in W^{1,2}(\X)$ and that the right hand side of \eqref{eq:chainhess} defines an object in $L^2((T^*)^{\otimes 2}\X)$. Now let  $g_1,g_2\in\test\X$, $h\in\LIP_b(\X)$ be arbitrary and notice that
\[
\begin{split}
-\la\nabla(\varphi\circ f),\nabla g_1\ra\,\div(h\nabla g_2)&=-\varphi'\circ f\la\nabla f,\nabla g_1\ra\,\div(h\nabla g_2)\\
&=-\la\nabla f,\nabla g_1\ra\,\div(\varphi'\circ fh\nabla g_2)+h\varphi''\circ f\la\nabla f,\nabla g_1\ra\la\nabla f,\nabla g_2\ra.
\end{split}
\]
Similarly,
\[
-\la\nabla(\varphi\circ f),\nabla g_2\ra\,\div(h\nabla g_1)=-\la\nabla f,\nabla g_2\ra\,\div(\varphi'\circ fh\nabla g_1)+h\varphi''\circ f\la\nabla f,\nabla g_2\ra\la\nabla f,\nabla g_1\ra
\]
and
\[
-h\big< \nabla(\varphi\circ f),\nabla\la \nabla g_1,\nabla g_2\ra\big>=-h\varphi'\circ f\big< \nabla  f,\nabla\la \nabla g_1,\nabla g_2\ra\big>.
\]
To conclude, add up these three identities, integrate and notice that   $h\varphi'\circ f\in \LIP_b(\X)$.
\end{proof}

\begin{proposition}[Product rule for gradients]\label{prop:gradehess}
Let $f_1,f_2\in \LIP\cap H^{2,2}(\X)$. Then $\la\nabla f_1,\nabla f_2\ra\in W^{1,2}(\X)$ and 
\begin{equation}
\label{eq:gradprod}
 \d\la\nabla f_1,\nabla f_2\ra=\He {f_1}(\nabla f_2,\cdot)+\He{f_2}(\nabla f_1,\cdot),\qquad\mm\ae.
\end{equation}
\end{proposition}
\begin{sketch} 
For $f_1,f_2\in\test\X$ the fact that $\la\nabla f_1,\nabla f_2\ra\in W^{1,2}(\X)$ follows from Theorem \ref{thm:savare} by polarization. Also, by the very definition of $H[f]$, we know that for any $g\in\test\X$ it holds
\[
\la\la\nabla f_1,\nabla f_2\ra,\nabla g\ra=H[f_1](f_2,g)+H[f_2](f_1,g),
\]
hence in this case the conclusion comes from \eqref{eq:hess12} and the arbitrariness of $g$. The general case follows by approximation by observing that with an argument based on truncation and regularization with the heat flow, we can approximate any $f\in \LIP\cap H^{2,2}(\X)$ in the $H^{2,2}(\X)$-topology with test functions which are uniformly Lipschitz.
\end{sketch}

\subsection{Covariant derivative}
\subsubsection{Sobolev vector fields}
The definition of Sobolev vector fields is based on the identity
\[
\la\nabla_{\nabla g_2}X,\nabla g_1\ra=\la\nabla(\la X,\nabla g_1\ra),\nabla g_2\ra-\He (g_1)(X,\nabla g_1),
\]
valid in the smooth world for smooth functions $g_1,g_2$ and a smooth vector field $X$.
\begin{definition}[The Sobolev space $W^{1,2}_C(T\X)$]\label{def:w12c}
The Sobolev space  $W^{1,2}_C(T\X)\subset  L^2(T\X)$ is the space of all $X\in L^2(T\X)$ for which there exists $T\in L^2(T^{\otimes 2}\X)$ such that for every $g_1,g_2\in\test\X$ and $h\in \LIP_b(\X)$ it holds
\[
\int h\, T:  (\nabla g_1\otimes \nabla g_2)\,\d\mm=\int-\la X,\nabla g_2\ra\,\div(h\nabla g_1)-h\He({ g_2})(X,\nabla g_1)\,\d\mm.
\]
In this case we shall call the tensor $T$ the covariant derivative of $X$ and denote it by $\nabla X$. We  endow $W^{1,2}_C(T\X)$ with the norm $\|\cdot\|_{W^{1,2}_C(T\X)}$ defined by
\[
\|X\|_{W^{1,2}_C(T\X)}^2:=\|X\|^2_{L^2(T\X)}+\|\nabla X\|_{L^2(T^{\otimes 2}\X)}^2.
\]
\end{definition}
It will be useful to introduce the space of `test vector fields' as
\[
\vsm:=\Big\{\sum_{i=1}^ng_i\nabla f_i\ :\ n\in\N,\ f_i,g_i\in \test\X\Big\}\subset L^2(T\X).
\]
It is easy to show that $\vsm$ is dense in $L^2(T\X)$.

\begin{theorem}[Basic properties of $W^{1,2}_C(T\X)$]\label{thm:basew12c}
We have:
\begin{itemize}
\item[i)] $W^{1,2}_C(T\X)$ is a separable Hilbert space.
\item[ii)] The covariant derivative is a closed operator, i.e. the set $\{(X,\nabla X) : X\in W^{1,2}_C(T\X)\}$ is a closed subset of $L^2(T\X)\times L^2(T^{\otimes 2}\X)$.
\item[iii)] Given  $f\in W^{2,2}(\X)$ we have $\nabla f\in W^{1,2}_C(T\X)$ with $\nabla(\nabla f)=(\He f)^\sharp$, where $\cdot^\sharp:L^2((T^*)^{\otimes 2}\X)\to L^2(T\X)$ is the Riesz (musical) isomorphims.
\item[iv)] We have $\vsm\subset W^{1,2}_C(T\X)$ with 
\[
\nabla X=\sum_i\nabla g_i\otimes\nabla f_i+g_i(\He {f_i})^\sharp,\qquad\text{ for }\qquad X=\sum_ig_i\nabla f_i.
\]
In particular, $W^{1,2}_C(T\X)$ is dense in $L^2(T\X)$.
\end{itemize}
\end{theorem}
\begin{sketch} $(i),(ii)$ are proved along the same lines of Theorem \ref{thm:basew2}.  $(iii)$ follows from  Proposition \ref{prop:gradehess} and direct verification; then $(iv)$ follows from $(iii)$ and the definitions.
\end{sketch}

\subsubsection{Calculus rules}\label{se:lc}
We know that $\vsm$ is contained in $W^{1,2}_C(T\X)$, but not if it is dense. Thus the following definition is meaningful:
\begin{definition}We define $H^{1,2}_C(T\X)\subset W^{1,2}_C(T\X)$ as the $W^{1,2}_C(T\X)$-closure of $\vsm$.
\end{definition}
We shall also denote by $L^0(T\X)$ the $L^0$-completion of $L^2(T\X)$ (Theorem \ref{thm:defl0}) and by $L^\infty(T\X)$ its subspace made of $X$'s such that $|X|\in L^\infty(\X)$.
\begin{proposition}[Leibniz rule]\label{prop:leibcov}
Let $X\in L^\infty\cap W^{1,2}_C(T\X)$ and $f\in L^\infty\cap W^{1,2}(\X)$.

Then  $fX\in W^{1,2}_C(T\X)$ and  
\begin{equation}
\label{eq:leibcov}
\nabla(fX)= \nabla f\otimes X+f\nabla X,\qquad\mm\ae.
\end{equation}
\end{proposition}
\begin{proof} Assume for the moment that $f\in\test\X$ and let  $g_1,g_2\in\test\X$, $h\in \LIP_b(\X)$ be arbitrary. Then  $fh\in \LIP_b(\X)$  and from the definition of $\nabla X$ we see that 
\[
\int fh\nabla X:(\nabla g_1\otimes\nabla g_2)\,\d\mm=\int-\la X, \nabla g_2\ra\,\div(fh\nabla g_1)- fh\,\He{g_2}(X,\nabla g_1)\,\d\mm.
\]
Using the identity   $\div(fh\nabla g_1)=h\la \nabla f,\nabla g_1\ra+f\div(h\nabla g_1)$ (recall \eqref{eq:leibdiv}), this gives
\[
\begin{split}
\int h\la\nabla f,\nabla g_1\ra\,\la X,\nabla g_2\ra + & fh\nabla X:(\nabla g_1\otimes \nabla g_2)\,\d\mm\\
&=\int -\la fX,\nabla g_2\ra\,\div(h\nabla g_1)- h\,\He{g_2}(fX,\nabla g_1)\,\d\mm,
\end{split}
\]
which is the thesis. The general case comes by approximation.
\end{proof}
It will be useful to introduce the following notation: for $X\in W^{1,2}_C(T\X)$ and $Z\in L^\infty(T\X)$, the vector field $\nabla_ZX\in L^2(T\X)$ is defined by
\[
\la \nabla_ZX, Y\ra:= \nabla X:(Z\otimes Y),\qquad\mm\ae,\qquad\forall Y\in L^2(T\X).
\]
Since $L^2(T\X)\ni Y\mapsto  \nabla X:(Z\otimes Y)\in L^1(\X)$ is continuous and $L^\infty$-linear, we see from Proposition \ref{prop:riesz} that this is a good definition.
\begin{proposition}[Compatibility with the metric]\label{prop:compmetr}
Let $X,Y\in L^\infty\cap H^{1,2}_C(T\X)$. Then $\la X, Y\ra \in W^{1, 2}(\X)$ and 
\[
\d\langle X,  Y\rangle(Z)=\la\nabla_Z X,  Y\ra+\la \nabla_ZY,X\ra,\qquad\mm\ae,
\]
for every $Z\in L^2(T\X)$.
\end{proposition}
\begin{sketch} For $X,Y\in\vsm$ the claim follows directly from \eqref{eq:gradprod} and \eqref{eq:leibcov}. The general case then follows by approximation (to be done carefully, because for $(X_n),(Y_n)$ converging to $X,Y$ in $H^{1,2}_C(T\X)$ the differential of $\la X_n,Y_n\ra$ only converge in $L^1(T^*\X)$ so that Proposition \ref{prop:closed} cannot be applied as it is).
\end{sketch}
In the following proposition and below we shall write $X(f)$ in place of $\d f(X)$. 
\begin{proposition}[Torsion free identity]
Let $f\in \LIP\cap H^{2,2}(\X)$ and $X,Y\in L^\infty\cap H^{1,2}_C(T\X)$. Then $X(f),Y(f)\in W^{1, 2}(\X)$ and 
\begin{equation}
\label{eq:torsionfree}
X(Y(f))-Y(X(f))=\d f(\nabla_XY-\nabla_YX),\qquad\mm\ae.
\end{equation}
\end{proposition}
\begin{proof} By the very definition of $H^{1,2}_C(T\X)$ we have $\nabla f\in  L^\infty\cap H^{1,2}_C(T\X)$, thus from Proposition \ref{prop:compmetr} we know that  $Y(f)\in W^{1,2}(\X)$ and 
\[
X(Y(f))=\nabla Y:(X\otimes \nabla f)+\He f(X,Y)=\d f(\nabla_XY)+\He f(X,Y).
\]
Subtracting the analogous expression for $Y(X(f))$ and using the symmetry of the Hessian we conclude.
\end{proof}
Since $\test\X\subset \LIP\cap H^{2,2}(\X)$, we have that $\{\d f: f\in \LIP\cap H^{2,2}(\X)\}$ generates $L^2(T^*\X)$, hence $\nabla_XY-\nabla_YX$ is the only vector field for which the identity \eqref{eq:torsionfree} holds. It is therefore meaningful to define the Lie bracket of vector fields as:
\[
[X,Y]:=\nabla_XY-\nabla_YX\in L^1(T\X)\qquad\forall X,Y\in H^{1,2}_C(T\X).
\]
\subsubsection{Flow of vector fields}
In the smooth setting, the Cauchy-Lipschitz theorem provides existence and uniqueness for the solution of
\begin{equation}
\label{eq:cl}
\gamma'_t=v_t(\gamma_t)\qquad\gamma_0\text{ given,}
\end{equation}
for a suitable family of Lipschitz vector fields $v_t$ on $\R^d$. The Ambrosio-Di Perna-Lions theory (\cite{DiPerna-Lions89}, \cite{Ambrosio04}) provides an extension of this classical result to the case of Sobolev/BV vector fields with a one-sided bound on the divergence. As it turned out (\cite{Ambrosio-Trevisan14}) such theory admits an extension to $\RCD$ spaces, which we very briefly recall here. We remark  that \cite{Ambrosio-Trevisan14} has been developed independently from \cite{Gigli14}, and that the definitions and results in \cite{Ambrosio-Trevisan14} cover cases more general than those we recall below: here we just want to phrase the main result of \cite{Ambrosio-Trevisan14} in the language we are proposing and in a set of assumptions which is usually relevant in applications.

\bigskip

The concept of solution of \eqref{eq:cl} is replaced by the following definition:
\begin{definition}[Regular Lagrangian flow]
Let $(X_t)\in L^2([0,1],L^2(T\X))$. We say that $F:[0,1]\times\X\to \X$ is a Regular Lagrangian Flow for $(X_t)$ provided:
\begin{itemize}
\item[i)] For some $C>0$ it holds
\begin{equation}
\label{eq:bdflow}
(F_t)_*\mm\leq C\mm\qquad\forall t\in[0,1].
\end{equation}
\item[ii)] For $\mm$-a.e.\ $x\in\X$ the curve $[0,1]\ni t\mapsto F_t(x)\in \X$ is continuous and such that $F_0(x)=x$.
\item[iii)] for every $f\in W^{1,2}(\X)$ we have: for $\mm$-a.e.\ $x\in\X$ the function $t\mapsto f(F_t(x))$ belongs to $W^{1,1}(0,1)$ and it holds
\begin{equation}
\label{eq:deffl}
\frac{\d}{\d t}f(F_t(x))=\d f(X_t)(F_t(x))\qquad \mm\times\mathcal L^1\restr{[0,1]}\ae (x,t)
\end{equation}
where the derivative at the left-hand-side is the distributional one.
\end{itemize}
\end{definition}
Notice that it is due to property $(i)$ that property $(iii)$ makes sense. Indeed, for given $X_t\in L^2(T\X)$ and $f\in W^{1,2}(\X)$ the function $\d f(X_t)\in L^1(\X)$ is only defined $\mm$-a.e., so that (part of) the role of  \eqref{eq:bdflow} is to grant that $\d f(X_t)\circ F_t$ is  well defined $\mm$-a.e..

Notice that by arguing as in the proof of the equality \eqref{eq:linksp12} we see that for $\mm$-a.e.\ $x\in \X$ the curve $t\mapsto F_t(x)$ is absolutely continuous with
\[
|\dot{F_t}(x)|=|X_t|(F_t(x))\qquad \mm\times\mathcal L^1\restr{[0,1]}\ae (x,t).
\]
Taking into account the integrability condition on $(X_t)$ we then see that for every $\mu\in\prob\X$ with $\mu\leq C\mm$ for some $C>0$, the plan $\ppi:=(F_\cdot)_*\mu$ is a test plan, where $F_\cdot:\X\to C([0,1],\X)$ is the $\mm$-a.e.\ defined map sending $x$ to $t\mapsto F_t(x)$. It is then clear from the defining properties \eqref{eq:vel12} and  \eqref{eq:deffl} that the velocity vector fields $\ppi'_t\in L^2(T\X,\e_t,\ppi)$ of $\ppi$ are given by
\[
\ppi'_t=\e_t^*X_t,\qquad{\rm a.e.}\ t.
\]
The main result of \cite{Ambrosio-Trevisan14} can then be stated as:
\begin{theorem}
Let $(X_t)\in L^2([0,1],W^{1,2}_C(T\X))\cap  L^\infty([0,1],L^\infty(T\X)) $ be such that $X_t\in D(\div)$ for a.e.\ $t\in[0,1]$, with 
\[
\int_0^1\|\div(X_t)\|_{L^2(\X)}+\|\big(\div(X_t)\big)^-\|_{L^\infty(\X)}\,\d t<\infty.
\]
Then a Regular Lagrangian Flow $F_t$ for $(X_t)$ exists and is unique, in the sense that if $\tilde F$ is another flow, then for $\mm$-a.e.\ $x\in \X$ it holds $F_t(x)=\tilde F_t(x)$ for every $t\in[0,1]$. Moreover:
\[
(F_t)_*\mm\leq \exp\Big(\int_0^t\|\big(\div(X_t)\big)^-\|_{L^\infty(\X)}\,\d t\Big)\,\mm\qquad\forall t\in[0,1].
\]
\end{theorem}
It is outside the scope of this note to present the proof of this result, which is non-trivial even in Euclidean setting; we rather refer to  \cite{AmbrosioCrippa08} and \cite{AT15} for an overview of the theory in $\R^n$ and $\RCD$ spaces respectively.

\subsection{Exterior derivative}
\subsubsection{Exterior power of a Hilbert module}
Let $\H$ be a Hilbert module and put $\H^{\otimes k}:=\underbrace{\H\otimes\cdots\otimes\H}_{k\text{ times}}$. The $k$-th exterior power $\H^{\wedge^k}$ of $\H$  is defined as the quotient of $\H^{\otimes k}$ w.r.t.\ the space of $L^\infty$-linear combinations of elements of the form $v_1\otimes\cdots\otimes v_k$ with $v_i=v_j$ for some $i\neq j$. 

We denote by $v_1\wedge\cdots\wedge v_k$ the image of $v_1\otimes\cdots\otimes v_k$ under the quotient map and  endow $\H^{\wedge^k}$ with the (rescaling of the) quotient pointwise scalar product given by
\[
\la v_1\wedge\cdots\wedge v_k,w_1\wedge\cdots\wedge w_k\ra:=\det\big(\langle v_i,w_j\rangle\big)\quad\mm\ae.
\]
Routine computations show that  $\H^{\wedge^k}$ is a Hilbert module. For $\H=L^2(T^*\X)$,  we write $L^2(\Lambda^kT^*\X)$ for the $k$-th exterior power if $k>1$, keeping the notation $L^2(T^*\X)$ and $L^2(\X)$ for the cases $k=1,0$ respectively. We shall refer to elements of $L^2(\Lambda^kT^*\X)$ as $k$-forms.

It is readily checked that the duality relation between $L^2(T^*\X)$ and $L^2(T\X)$ induces a duality relation between the respective $k$-th exterior powers; we shall typically write
$\omega(X_1,\ldots,X_k)$ in place of $\omega(X_1\wedge\cdots\wedge X_k)$.
\subsubsection{Sobolev differential forms and basic calculus rules}\label{se:extd}
In the smooth setting the exterior differential of the $k$-form $\omega$ if characterized by
\[
\begin{split}
\d\omega(X_0,\ldots,X_k)=&\sum_i(-1)^i\d\big(\omega(X_0,\ldots,\hat X_i,\ldots,X_k)\big)(X_i)\\
&+\sum_{i<j}(-1)^{i+j}\omega([X_i,X_j],X_0,\ldots,\hat X_i,\ldots,\hat X_j,\ldots,X_k),
\end{split}
\]
for any smooth vector fields $X_1,\ldots,X_k$.

Noticing that  for $X_i\in\vsm$ we have  $|X_1\wedge\ldots\wedge X_n|\in L^2(\X)$ and   $|[X_i,X_j]\wedge  X_1\wedge\ldots\wedge X_n|\in L^2(\X)$ as well, we are therefore lead to the following definition:
\begin{definition}[The space $W^{1,2}_\d(\Lambda^kT^*\X)$]
The space $W^{1,2}_\d(\Lambda^kT^*\X)\subset L^2(\Lambda^kT^*\X)$ is the space of $k$-forms $\omega$ such that there exists a $k+1$ form $\eta\in L^2(\Lambda^{k+1}T^*\X)$ for which the identity
\begin{equation}
\label{eq:defdext}
\begin{split}
\int \!\eta(X_0,\cdots, X_{k})\,\d\mm=& \int\sum_i (-1)^{i+1}\omega(X_0,\cdots,\hat X_i,\cdots, X_{k})\,\div( X_i)\,\d\mm\\
&+\int\sum_{i<j}(-1)^{i+j}\omega([X_i,X_j],X_0,\cdots ,\hat X_i ,\cdots ,\hat X_j ,\cdots, X_{k})\,\d\mm,
\end{split}
\end{equation}
holds for any $X_0,\ldots,X_{k}\in \vsm$. In this case $\eta$ will be called exterior differential of $\omega$ and denoted as $\d\omega$.

We endow $W^{1,2}_\d(\Lambda^kT^*\X)$ with the norm $\|\cdot\|_{W^{1,2}_\d(\Lambda^kT^*\X)}$ given by
\[
\|\omega\|_{W^{1,2}_\d(\Lambda^kT^*\X)}^2:=\|\omega\|^2_{L^2(\Lambda^kT^*\X)}+\|\d\omega\|^2_{L^2(\Lambda^{k+1}T^*\X)}.
\]
\end{definition}
It is readily verified that for $\omega\in W^{1,2}_\d(\Lambda^kT^*\X)$ the $\eta$ for which \eqref{eq:defdext} holds is unique and linearly depends on $\omega$, so that $W^{1,2}_\d(\Lambda^kT^*\X)$ is a normed vector space.

We then have the following:
\begin{theorem}[Basic properties of $W^{1,2}_\d(\Lambda^kT^*\X)$]\label{thm:basew12d} For every $k\in\N$ the following holds:
\begin{itemize}
\item[i)]  $W^{1,2}_\d(\Lambda^kT^*\X)$ is a separable Hilbert space.
\item[ii)] The exterior differential is a closed operator, i.e.\ $\{(\omega,\d\omega):\omega\in W^{1,2}_\d(\Lambda^kT^*\X)\}$ is a closed subspace of $L^{2}(\Lambda^kT^*\X)\times L^{2}(\Lambda^{k+1}T^*\X)$.
\item[iii)]  $W^{1,2}_\d(\Lambda^0T^*\X)=W^{1,2}(\X)$ and the two notions of differentials underlying these spaces coincide. 
\end{itemize}
\end{theorem}
\begin{proof} $(i)$ and $(ii)$ are proved along the same lines used for analogous claims in Theorem \ref{thm:basew2}. For $(iii)$ we notice that the inclusion $\supset$ and the fact that for $f\in W^{1,2}(\X)$ its differential as defined in Theorem \ref{thm:defcot} satisfies \eqref{eq:defdext} is obvious by the very definition of divergence. For the converse inclusion notice that in the case $k=0$ \eqref{eq:defdext} reads as
\[
-\int f\div(X)\,\d\mm=\int \eta(X)\,\d\mm\qquad\forall X\in \vsm,
\]
let $(f_n)\subset L^2\cap L^\infty(\X)$ be $L^2$-converging to $f$ and notice that for $t>0$ we have $\nabla \h_tf_n\in\vsm$, so that the above holds for $X=\nabla \h_tf_n$. Passing to the limit in $n$ and noticing that $\nabla \h_tf_n\to\nabla\h_t f$ and $\Delta \h_tf_n\to\Delta \h_tf$ in $L^2(T\X)$ and $L^2(\X)$ respectively we deduce
\[
\int|\nabla\h_{t/2}f|^2\,\d\mm=-\int f\Delta\h_tf(X)\,\d\mm=\int \eta(\nabla\h_tf)\,\d\mm\leq \|\eta\|_{L^2}\|\nabla\h_t f\|_{L^2}\leq \|\eta\|_{L^2}\|\nabla\h_{t/2} f\|_{L^2},
\]
having used the fact that $t\mapsto \int|\nabla\h_{t/2}f|^2\,\d\mm$ is non-increasing. The conclusion follows dividing by $\|\nabla\h_{t/2} f\|_{L^2}$ and letting $t\downarrow0$.
\end{proof}
It will be convenient to introduce the space of {\bf test $k$-forms} as
\[
\begin{split}
\ffsm k:=\Big\{\text{linear combinations of forms of the}& \text{ kind }f_0\d f_1\wedge\ldots\wedge\d f_k\\
&\text{with }\ f_i\in\test\X\ \forall i=0,\ldots,k\Big\}.
\end{split}
\]
It is not hard to check that $\ffsm k$ is dense in $L^2(\Lambda^kT^*\X)$.

\begin{proposition}[Basic calculus rules for exterior differentiation]\label{prop:basedext} The following holds:
\begin{itemize}
\item[i)] For  $f_i\in L^\infty\cap W^{1,2}(\X)$ with $|\d f_i|\in L^\infty$, $i=0,\ldots,k$, we have that both  $ f_0\d f_1\wedge\cdots\d f_k$ and $\d f_1\wedge\cdots\d f_k$ are in $W^{1,2}_\d(\Lambda^kT^*X)$ and
\begin{align}
\label{eq:df1fn}
\d( f_0\d f_1\wedge\cdots \wedge\d f_k)&= \d f_0\wedge \d f_1\wedge\cdots \wedge\d f_k,\\
\label{eq:ddf1fn}
\d(\d f_1\wedge\cdots \wedge\d f_k)&= 0.
\end{align}
\item[ii)]  We have $\ffsm k\subset W^{1,2}_\d(\Lambda^kT^*\X)$ and in particular $W^{1,2}_\d(\Lambda^kT^*\X)$ is dense in $L^{2}(\Lambda^kT^*\X)$.
\item[iii)]Let $\omega\in W^{1,2}_\d(\Lambda^{k}T^*\X)
$ and $\omega'\in\ffsm{k'}$. Then $\omega\wedge\omega'\in W^{1,2}_\d(\Lambda^{k+k'}T^*\X)$ with
\[
\d(\omega\wedge\omega')=\d\omega\wedge\omega'+(-1)^k\omega\wedge\d\omega'.
\]
\end{itemize}
\end{proposition}
\begin{sketch} These all follow from the definitions, the identity $\d f_1\wedge\ldots\wedge\d f_k(X_1,\ldots,X_k)=\det(\d f_i(X_j))$ and routine computations based on   the calculus rules obtained so far.
\end{sketch}
This last proposition motivates the following definition:
\begin{definition}
$H^{1,2}_\d(\Lambda^kT^*\X)\subset W^{1,2}_\d(\Lambda^kT^*\X)$ is the $W^{1,2}_\d$-closure of $\ffsm k$.
\end{definition}
Clearly,  $H^{1,2}_\d(\Lambda^kT^*\X)$ is dense in  $L^{2}(\Lambda^kT^*\X)$. Another crucial property of $H^{1,2}_\d$-forms is:
\begin{proposition}[$\d^2=0$ for forms in  $H^{1,2}_\d(\Lambda^kT^*\X)$]\label{prop:dd}
Let  $\omega\in H^{1,2}_\d(\Lambda^kT^*\X)$. Then
\[
\d\omega\in H^{1,2}_\d(\Lambda^{k+1}T^*\X)\qquad\text{and}\qquad \d(\d\omega)=0.
\]
\end{proposition}
\begin{proof} The identities \eqref{eq:df1fn} and \eqref{eq:ddf1fn} establish the claim for forms in $\ffsm k$. The general case then follows  by approximation taking into account the closure of the exterior differential.
\end{proof}
\subsubsection{de Rham cohomology  and Hodge theorem}\label{se:dr}
Proposition \ref{prop:dd} is the starting point for building de Rham cohomology. The definition of closed and exact $k$-forms is naturally given by:
\[
\begin{split}
\closed k&:=\big\{\omega\in H^{1,2}_\d(\Lambda^kT^*\X)\ :\ \d\omega=0\big\},\qquad\qquad\exact k:=\big\{\d\omega\ :\ \omega\in H^{1,2}_\d(\Lambda^{k-1}T^*\X)\big\}.
\end{split}
\]
Proposition \ref{prop:dd} ensures that $\exact k\subset\closed k$ and the closure of the differential that  $\closed k$ is a closed subspace of $L^2(\Lambda^kT^*\X)$. Hence defining $\exacto k$ as 
\[
\exacto k:=\textrm{ $L^2(\Lambda^kT^*\X)$-closure of }\exact k
\]
we also have that $\exacto k\subset \closed k$. We can then give the following:
\begin{definition}[de Rham cohomology]
For $k\in \N$ the Hilbert space $H^k_{dR}(\X)$ is defined as the quotient
\[
{\sf H}^k_{\sf dR}(\X):=\frac{\closed k}{\exacto k},
\]
where $\closed k$ and $\exacto k$ are endowed with the $L^2(\Lambda^kT^*\X)$-norm.
\end{definition}
Cohomology as we just defined it is  functorial in the following sense. Let   $\varphi:\X_2\to\X_1$ be of bounded deformation and recall that in Theorem \ref{thm:pbf} we gave the definition of pullback of 1-forms $\varphi^*:L^2(T^*\X_1)\to L^2(T^*\X_2)$. It is then not hard  to see that for every $k\in\N$ there is a unique linear map $\varphi^*:L^2(\Lambda^kT^*\X_1)\to  L^2(\Lambda^kT^*\X_2)$ such that
\begin{equation}
\label{eq:pbk}
\begin{split}
\varphi^*(\omega_1\wedge\ldots\wedge \omega_k)&=(\varphi^*\omega_1)\wedge\ldots\wedge (\varphi^*\omega_k),\\
\varphi^*(f\omega)&=f\circ\varphi\,\varphi^*\omega,\\
|\varphi^*\omega|&\leq \lf(\varphi)^k|\omega|\circ\varphi,
\end{split}
\end{equation}
for every $\omega_1,\ldots,\omega_k\in L^2\cap L^\infty(T^*\X_1)$, $\omega\in L^2(\Lambda^kT^*\X_1)$ and $f\in L^\infty(\X_2)$.

Then we have:
\begin{proposition}[Functoriality]
Let $(\X_1,\sfd_1,\mm_1),(\X_2,\sfd_2,\mm_2)$ be two $\RCD(K,\infty)$ spaces, $K\in\R$, and $\varphi:\X_2\to\X_1$ of bounded deformation. Then for every $k\in\N$ and $\omega\in H^{1,2}_\d(\Lambda^kT^*\X_1)$ we have $\varphi^*\omega\in H^{1,2}_\d(\Lambda^kT^*\X_2)$ and
\begin{equation}
\label{eq:functd}
\d(\varphi^*\omega)=\varphi^*\d\omega.
\end{equation}
In particular, $\varphi^*$ passes to the quotient and induces a linear continuous map from ${\sf H}^k_{\sf dR}(\X_1)$ to ${\sf H}^k_{\sf dR}(\X_2)$ with norm bounded by $\lf(\varphi)^k$.
\end{proposition}
\begin{proof}
From the linearity and continuity of $\varphi^*$ and of  $\d: H^{1,2}_\d(\Lambda^kT^*\X_2)\to L^2(\Lambda^{k+1}T^*\X_2)$, it is sufficient to prove \eqref{eq:functd} for $\omega$ of the form $\omega=f_0\d f_1\wedge\ldots\wedge\d f_k$, for $f_i\in\test{\X_1}$. In this case \eqref{eq:pbk} gives that 
\[
\varphi^*\omega=f_0\circ\varphi\,\d(f_1\circ\varphi)\wedge\ldots\wedge\d(f_k\circ\varphi)
\]
and since $f_i\circ\varphi\in L^\infty\cap W^{1,2}(\X_2)$ with $|\d(f_i\circ\varphi)|\in L^\infty(\X_2)$, from point $(i)$ of Proposition \ref{prop:basedext} we deduce that
\[
\d\varphi^*\omega=\d(f_0\circ\varphi)\wedge \d(f_1\circ\varphi)\wedge\ldots\wedge\d(f_k\circ\varphi)=\varphi^*\d\omega,
\]
as desired.

The fact that $\varphi^*$ passes to the quotient is then a direct consequence of its linearity and continuity, and  the bound on the norm comes directly from the last in \eqref{eq:pbk}.
\end{proof}
We now want to show that an analogue of Hodge theorem about representation of cohomology classes via harmonic forms holds.  We shall need a few definitions.

We start with that of {\bf codifferential}, defined as the adjoint of the exterior differential: for $k\in\N$ the space $D(\delta)\subset L^2(\Lambda^kT^*\X)$ is the space of those forms $\omega$ for which there exists a form $\delta\omega\in L^2(\Lambda^{k-1}T^*\X)$, called codifferential of $\omega$, such that 
\[
\int \la \delta\omega,\eta\ra\,\d\mm=\int \la\omega,\d\eta\ra\,\d\mm,\qquad\forall\eta\in\ffsm{k-1}.
\]
In  the case $k=0$ we put $D(\delta_0):=L^2(\X)$ and define the $\delta$ operator to be identically 0 on it.

It is not hard to check that $\delta$ is well-defined and closed, while some computations (which we omit) show that $\ffsm k\subset D(\delta)$. In particular, the following definitions of `Hodge' Sobolev spaces are meaningful:
\begin{definition} For $k\in\N$,  we define $W^{1,2}_\Ho(\Lambda^kT^*\X):=W^{1,2}_\d(\Lambda^kT^*\X)\cap D(\delta)$  with the norm
\[
\|\omega\|_{W^{1,2}_\Ho(\Lambda^kT^*\X)}^2:=\|\omega\|^2_{L^2(\Lambda^kT^*\X)}+\|\d\omega\|_{L^2(\Lambda^{k+1}T^*\X)}^2+\|\delta\omega\|^2_{L^2(\Lambda^{k-1}T^*\X)}.
\]
The space $H^{1,2}_\Ho(\Lambda^kT^*\X)$ is the $W^{1,2}_\Ho$-closure of $\ffsm k$. 
\end{definition}
In particular, $H^{1,2}_\Ho(\Lambda^kT^*\X)$ is a Hilbert space dense in $L^{2}(\Lambda^kT^*\X)$.

\begin{definition}[Hodge Laplacian and harmonic forms]
Given $k\in\N$, the domain $D(\Delta_{\Ho})\subset H^{1,2}_\Ho(\Lambda^kT^*\X)$ of the Hodge Laplacian is the set of $\omega\in H^{1,2}_\Ho(\Lambda^kT^*\X)$   for which there exists $\alpha\in L^2(\Lambda^kT^*\X)$ such that
\[
\int\la\alpha,\eta\ra\,\d\mm=\int \la \d\omega,\d\eta\ra+\la \delta\omega,\delta\eta\ra\,\d\mm,\qquad\forall \eta\in H^{1,2}_\Ho(\Lambda^kT^*\X).
\]
In this case, the form $\alpha$ (which is unique by the density of $H^{1,2}_\Ho(\Lambda^kT^*\X)$ in $L^2(\Lambda^kT^*\X)$) will be called Hodge Laplacian of $\omega$ and denoted by $\Delta_\Ho\omega$.

The space $\harm k\subset D(\Delta_{\Ho})$ is the space of forms $\omega\in D(\Delta_{\Ho})$ such that $\Delta_\Ho\omega=0$.
\end{definition}
In the case of functions, we have the usual  unfortunate sign relation:
\[
\Delta_\Ho f=-\Delta f\qquad\forall f\in D(\Delta)=D(\Delta_\Ho)\subset L^2(\Lambda^0T^*\X)=L^2(\X).
\]
The Hodge Laplacian is a closed operator: this can be seen by noticing that it is the subdifferential of the convex and lower semicontinuous functional on $L^2(\Lambda^kT^*\X)$ defined by
\[
\omega\quad\mapsto\quad \frac12\int |\d\omega|^2+|\delta\omega|^2\,\d\mm\qquad \text{if $\omega\in  H^{1,2}_\Ho(\Lambda^kT^*\X)$, \qquad $+\infty$ otherwise}.
\]
From such closure it follows that   $\harm k$  is a closed subspace of $L^2(\Lambda^kT^*\X)$ and thus a Hilbert space itself when endowed with the $L^2(\Lambda^kT^*\X)$-norm. We then have:
\begin{theorem}[Hodge theorem on $\RCD$ spaces]\label{thm:hodge} The map
\[
\harm k\ni \omega\qquad\mapsto\qquad[\omega]\in {\sf H}^k_{\sf dR}(\X)
\]
is an isomorphism of Hilbert spaces.
\end{theorem}
\begin{proof} Start noticing that 
\[
\omega\in\harm k\qquad\Leftrightarrow\qquad \d \omega=0\ \text{ and }\ \delta\omega=0,
\] 
indeed the `if' is obvious by definition, while the `only if' comes from the identity 
\[
\int\la\omega,\Delta_\Ho\omega\ra\,\d\mm=\int|\d\omega|^2+|\delta\omega|^2\,\d\mm.
\]
Recalling the definition of $\delta$, we thus see that 
\[
\omega\in\harm k\qquad\Leftrightarrow\qquad  \omega\in \closed k\ \text{ and }\ \int\la\omega,\eta\ra\,\d\mm=0\quad \forall\eta\in\exact k.
\]
The conclusion follows recalling that for every Hilbert space $H$ and subspace $V$, the map
\[
V^\perp\ni w\quad\mapsto\quad w+\overline V\in H/\overline V
\]
is an isomorphism of Hilbert spaces.
\end{proof}

\subsection{Ricci curvature}\label{se:ricci}
In the course of this section we shall abuse a bit the notation and identify vector and covector fields, thus in for instance we shall write $X\in D(\Delta_\Ho)$ and consider the vector field $\Delta_\Ho X\in L^2(T\X)$ when we should write $X^\flat \in D(\Delta_\Ho) $ and  $(\Delta_\Ho X^\flat)^\sharp\in L^2(T\X)$, where $\cdot^\flat:L^2(T\X)\to L^2(T^*\X)$ and $\cdot^\sharp:L^2(T^*\X)\to L^2(T\X)$ are the Riesz (musical) isomorphisms. 

We begin reinterpreting  the key Lemma \ref{le:lemmachiave}:  the differential operators introduced  so far allow to restate the key inequalities \eqref{eq:partesing1}, \eqref{eq:parteac1}  in a much more familiar way.
\begin{lemma}\label{le:riscritto}
Let $X\in\vsm$. Then $X\in D(\Delta_{\Ho})$, $|X|^2\in D(\bd)$ and the inequality
\begin{equation}
\label{eq:bochner}
\bd\frac{|X|^2}{2}\geq \Big(|\nabla X|_\HS^2-\la X, \Delta_\Ho X \ra+K|X|^2\Big)\mm
\end{equation}
holds
\end{lemma}
\begin{sketch} Let $X=\sum_ig_i\nabla f_i$ for $f_i,g_i\in\test\X$. It is only a matter of computations to see that $|X|^2\in D(\bd)$ and $X\in D(\Delta_\Ho)$ with
\[
\begin{split}
\bd\frac{|X|^2}{2}&=\sum_{i,j}\frac12 g_i  g_j\bd\la\nabla f_i,\nabla f_j\ra+\Big(g_j\Delta g_i\la \nabla f_i,\nabla f_j\ra+\la\nabla g_i,\nabla g_j\ra\langle\nabla f_i,\nabla f_j\rangle\Big)\mm\\
&\qquad\qquad\qquad\qquad\qquad\qquad +\Big(2g_i\He{f_i}(\nabla f_j,\nabla g_j)+2g_i\He{f_j}(\nabla f_i,\nabla g_j)\Big)\mm\\
\Delta_{\Ho}X&=\sum_i-g_i\d\Delta f_i-\Delta g_i\d f_i-2\He {f_i}(\nabla g_i,\cdot)\\
(\nabla X)_{\sf Asym}&=\sum_{i}\frac{\nabla g_i\otimes\nabla f_i-\nabla f_i\otimes\nabla g_i}2
\end{split}
\]
and thus recalling the definition of the measure $\mu((f_i),(g_i))$ given in Lemma \ref{le:lemmachiave} we see that
\[
\mu\big((f_i),(g_i)\big)=\bd\frac{|X|^2}{2}+\Big(\la X, (\Delta_\Ho X)\ra-K|X|^2-|(\nabla X)_{\sf Asym}|_\HS^2\Big)\mm.
\]
Therefore writing  $\bd\frac{|X|^2}{2}=\Delta_{ac}\frac{|X|^2}{2}\mm+\bd_{sing}\frac{|X|^2}{2}$, with $\bd_{sing}\frac{|X|^2}{2}\perp\mm$, inequality \eqref{eq:partesing1} in Lemma \ref{le:lemmachiave} yields
\begin{equation}
\label{eq:bsb}
\bd_{sing}\frac{|X|^2}{2}\geq 0
\end{equation}
while from \eqref{eq:parteac1}  we see that for every $m\in\N$ and choice of $h_1,\ldots,h_m\in\test\X$ we have
\[
\left|\nabla X:\sum_{i=1}^m\nabla h_i\otimes\nabla h_i\right|\leq  \sqrt{ \Delta_{ac}\frac{|X|^2}{2}+ \la X, \Delta_\Ho X\ra-K|X|^2-|(\nabla X)_{\sf Asym}|_\HS^2}\left|\sum_{i=1}^m\nabla h_i\otimes\nabla h_i\right|_\HS
\]
$\mm$-a.e., which in turn implies
\[
2\nabla X:\sum_{i=1}^m\nabla h_i\otimes\nabla h_i-\left|\sum_{i=1}^m\nabla h_i\otimes\nabla h_i\right|_\HS^2\leq \Delta_{ac}\frac{|X|^2}{2}+  \la X, \Delta_\Ho X\ra-K|X|^2-|(\nabla X)_{\sf Asym}|_\HS^2
\]
$\mm$-a.e.. Noticing that $L^\infty$-linear combinations of objects of the form $\nabla h\otimes\nabla h$ for $h\in\test\X$, are $L^2$-dense in the space of symmetric 2-tensors, taking the (essential) supremum in this last inequality among  $m\in\N$ and choices of $h_1,\ldots,h_m\in\test\X$ we obtain
\[
|(\nabla X)_{\sf Sym}|_\HS^2\leq \Delta_{ac}\frac{|X|^2}{2}+ \la X,\Delta_\Ho X\ra-K|X|^2-|(\nabla X)_{\sf Asym}|_\HS^2,\qquad\mm\ae,
\]
which, recalling \eqref{eq:normasym} and \eqref{eq:bsb}, gives the conclusion.
\end{sketch}
Let us introduce the `covariant energy' and the `Hodge energy' functionals on $L^2(T\X)$ as 
\begin{align*}
\ec(X)&:=\frac12\int|\nabla X|^2\,\d\mm\qquad& \text{ if }X\in H^{1,2}_C(T\X),\quad +\infty\text{ otherwise},\\
\eh(X)&:=\frac12\int|\d X|^2+|\delta X|^2\,\d\mm\qquad& \text{ if }X\in H^{1,2}_\Ho(T\X),\quad+\infty\text{ otherwise}.
\end{align*}
Notice that the closure of the differential operators involved grant that these are $L^2(T\X)$-lower semicontinuous.
Then  the last lemma has the following useful corollary (which generalizes Corollary \ref{cor:bello}):
\begin{corollary}\label{cor:bello2}
We have $H^{1,2}_\Ho(T\X)\subset H^{1,2}_C(T\X)$ and
\begin{equation}
\label{eq:ehec}
\ec(X)\leq \eh (X)-\frac K2\|X\|^2_{L^2(T\X)},\qquad\forall X\in H^{1,2}_\Ho(T\X).
\end{equation}
\end{corollary}
\begin{proof} For  $X\in\vsm$ the bound \eqref{eq:ehec} comes integrating \eqref{eq:bochner} recalling \eqref{eq:deltal1}. The general case then follows approximating $X\in H^{1,2}_\Ho(T\X)$ with vector fields in $\vsm$ and using the $L^2$-lower semicontinuity of $\ec$.
\end{proof}
We are now ready to introduce the Ricci curvature operator:
\begin{thmdef}[Ricci curvature]\label{thm:ricci}
There exists a unique continuous map, called Ricci curvature,  $\ric: [H^{1,2}_\Ho(T\X)]^2\to \mes(\X)$ such that for every $X,Y\in\vsm$ it holds
\begin{equation}
\label{eq:defricci}
\ric(X,Y)=\bd\frac{\la X, Y\ra}2+\Big(\frac12\la X,\Delta_\Ho Y \ra+ \frac12\la Y,\Delta_\Ho X\ra-\nabla X:\nabla Y\Big)\mm.
\end{equation}
Such map is bilinear, symmetric and satisfies
\begin{align}
\label{eq:riccibound}
\ric(X,X)&\geq K|X|^2\mm\\
\label{eq:riccitotal}
\ric(X,Y)(\X)&=\int\la\d X,\d Y\ra+\delta X\,\delta Y-\nabla X:\nabla Y\,\d\mm\\
\label{eq:riccitv}
\|\ric(X,Y)\|_{\sf TV}&\leq2\sqrt{\eh(X)+K^-\|X\|_{L^2(T\X)}^2}\,\sqrt{\eh(Y)+K^-\|Y\|_{L^2(T\X)}^2}
\end{align}
for every $X,Y\in H^{1,2}_\Ho(T\X)$, where $K^-:=\max\{0,-K\}$.
\end{thmdef}
\begin{sketch} The fact that the right hand side of \eqref{eq:defricci} is well defined for  $X,Y\in\vsm$  is a direct consequence of Lemma \ref{le:riscritto}. That such right hand side is bilinear, symmetric and satisfies \eqref{eq:riccitotal} is obvious, while \eqref{eq:riccibound} is a restatement of   \eqref{eq:bochner}. Thanks to the density of $\vsm$ in $H^{1,2}_\Ho(T\X)$, to conclude it is therefore sufficient to prove \eqref{eq:riccitv} for $X,Y\in\vsm$:  we shall do so for the case $K=0$ only.

Let $X,Y\in\vsm$, choose $\mu\in\mes(\X)$, $\mu\geq 0$, such that $\ric(X,X),\ric(X,Y)$ and $\ric(Y,Y)$ are all absolutely continuous w.r.t.\ $\mu$ and let $f,g,h$ be the respective Radon-Nikodym derivatives. Then \eqref{eq:riccibound} grants that $f,h\geq 0$ $\mu$-a.e. and that for any $\lambda\in\R$ we have $\ric(\lambda X+Y,\lambda X+Y)\geq 0$. Hence
\[
\lambda^2 f+2\lambda g+h\geq 0,\qquad \mu\ae,
\]
which easily implies $|g|\leq \sqrt{fh}$ $\mu$-a.e.\ and therefore
\[
\|\ric(X,Y)\|_{\sf TV}=\int |g|\,\d\mu\leq \sqrt{\int f\,\d\mu \int h\,\d\mu} =\sqrt{\|\ric(X,X)\|_{\sf TV}\|\ric(Y,Y)\|_{\sf TV}}.
\]
The conclusion then follows noticing that
\[
\|\ric(X,X)\|_{\sf TV}\stackrel{\eqref{eq:riccibound}}=\ric(X,X)(\X)\stackrel{\eqref{eq:riccitotal}}=2\eh(\X)-2\ec(\X)\leq 2\eh(\X)\qquad\forall X\in\vsm.
\]
\end{sketch}

The Ricci curvature operator as defined in the last theorem is a tensor in the sense that it holds:
\[
\ric(fX,Y)=f\,\ric(X,Y)\qquad\forall X,Y\in H^{1,2}_\Ho(T\X),\ f\in\test\X,
\]
as can be showed with some algebraic manipulations based on the calculus rules developed so far (we omit the details). Moreover, directly from the definitions we get
\[
\left.\begin{array}{l}
(\X,\sfd,\mm)\text{ is a $\RCD(K',\infty)$ space with}\\
\ric(X,X)\geq K|X|^2\mm\quad\forall X\in H^{1,2}_\Ho(T\X)
\end{array}\right\}\qquad\Rightarrow\qquad  (\X,\sfd,\mm)\text{ is a $\RCD(K,\infty)$ space.}
\]
\begin{remark}{\rm
Directly from the definition it is easy to see that the Ricci measure gives  0 mass to sets with 0 capacity.  It follows that, for instance, on a two dimensional space with a conical singularity, the Ricci curvature as we defined it does not see any `delta' at the vertex: this also implies that we cannot hope for such measure to have any kind of Gauss-Bonnet formula.

If the space is sufficiently regular ($C^{1,1}$ manifold is enough), then one can detect the singularity of the curvature at the vertex of such a cone by computing the curvature along objects more regular than Sobolev vector fields, namely Lipschitz half densities (see \cite{Lott2016}).
}\fr\end{remark}
\subsection{Some properties in the finite dimensional case}\label{se:findim}
Here we briefly present, without proofs, some related results about analysis and geometry of finite dimensional $\RCD$ spaces (\cite{Gigli12}, \cite{AmbrosioGigliSavare12},  \cite{Erbar-Kuwada-Sturm13}, \cite{AmbrosioMondinoSavare13}).
\begin{definition}[$\RCD^*(K,N)$ spaces]
Let $K\in\R$, $N\in[1,\infty)$. $(\X,\sfd,\mm)$ is a $\RCD^*(K,N)$ space provided it is a $\RCD^*(K,\infty)$ space and the Bochner inequality holds in the following form:
\[
\frac12\int\Delta g|\d f|^2\,\d\mm\geq \int g\Big(\frac{(\Delta f)^2}N+\la\nabla f,\nabla \Delta f\ra+K|\d f|^2\Big)\,\d\mm
\]
for every $f\in D(\Delta)$ with $\Delta f\in W^{1,2}(\X)$ and $g\in L^\infty(\X)\cap D(\Delta)$ with $g\geq 0$, $\Delta g\in L^\infty(\X)$.
\end{definition}
On compact finite-dimensional $\RCD$ spaces, the following natural second-order differentiation formula holds (proved in \cite{GT17}), which links the Hessian as we defined it to the second derivative along geodesics, compare with Theorem \ref{thm:dergeo}.
\begin{theorem}[Second order differentiation formula]
Let $(\X,\sfd,\mm)$ be a compact $\RCD^*(K,N)$ space, $N<\infty$ and $(\mu_t)\subset \prob\X$ a $W_2$-geodesic such that $\mu_0,\mu_1\leq C\mm$ for some $C>0$.

Then for every $f\in H^{2,2}(\X)$ the map $t\mapsto\int f\,\d\mu_t$ is $C^2([0,1])$ and it holds
\[
\frac{\d^2}{\d t^2}\int f\,\d\mu_t=\int\He f(\nabla\varphi_t,\nabla\varphi_t)\,\d\mu_t\qquad\forall t\in[0,1],
\]
where $\varphi_t$ is, for every $t\in[0,1]$,  such that for some $s\neq t$ the function $(s-t)\varphi$ is a Kantorovich potential from $\mu_t$ to $\mu_s$.
\end{theorem}
The proof of this theorem relies upon an approximation of $W_2$-geodesics with the so-called {\bf entropic interpolation} (see \cite{Leonard14} for an overview on the topic). The result requires finite-dimensionality because is based, among other things, on the Li-Yau inequality.  Compactness is likely not needed, but so far the general result is unknown.

A better understanding of the structure of $\RCD$ spaces can be achieved by introducing the concept of {\bf local dimension} of a module: we say that $\M$ has dimension $n\in\N$ on the Borel set $E\subset\X$ provided there are $v_1,\ldots,v_n\in\M$ such that 
\[
\begin{split}
&\sum_if_iv_i=0\qquad\Rightarrow \qquad f_i=0\quad\mm\ae\ \text{on $E$ for every }i=1,\ldots,n,\\
&\text{$L^\infty$-linear combinations of the $v_i$'s are dense in }\{v\in\M\ :\ \nchi_{E^c}v=0\}.
\end{split}
\]
It is then not hard to see that for any given module there exists a (unique up to negligible sets) Borel partition $(E_i)_{i\in\N\cup\{\infty\}}$ of $\X$ such that $\M$ has dimension $i$ on $E_i$ for every $i\in\N$ and has not finite dimension on any $F\subset E_\infty$ with positive measure.

When the module under consideration is the tangent one, we call the resulting partition the dimensional decomposition of $\X$. This also allows to $\mm$-a.e.\ define the `analytic local dimension' function $\dim_{\rm loc}:\X\to \N\cup\{\infty\}$ which sends $E_i$ to $i$ for every $i\in\N\cup\{\infty\}$. It is conjectured that such function is actually constant (after \cite{ColdingNaber12} and Theorem \ref{thm:GP} below this is known to hold at least for Ricci-limit spaces), but so far this is unknown.

The results in \cite{Mondino-Naber14} grant that the pointed rescaled spaces $(\X,\sfd/r,\mm(B_r(x))^{-1}\mm,x)$ converge, for $\mm$-a.e.\ $x\in\X$,   to the Euclidean space $(\R^{n(x)},\sfd_{\rm Eucl},\mathcal L^{n(x)},0)$ in the pointed-measured-Gromov-Hausdorff sense for some $n(x)\in\N$, $n(x)\leq N$. In particular, the number $n(x)$ provides a `geometric' notion of dimension at $x$. It turns out (\cite{GP16}) that this notion is equivalent to the analytic one:
\begin{theorem}\label{thm:GP}
With the above notation, we have $\dim_{\rm loc}(x)=n(x)$ for $\mm$-a.e.\ $x\in\X$. In particular, $\mm$-a.e.\ we have $\dim_{\rm loc}\leq N$.
\end{theorem}
In fact, something stronger holds: the tangent module $L^2(T\X)$ is isomorphic to the space of `$L^2$ sections' of the bundle on $\X$ made of the collections of the pmGH-limits of rescaled spaces. The proof of this fact uses the {\bf charts} built in  \cite{Mondino-Naber14}, along with the improvements given in \cite{MK16} and \cite{GP16-2}, to produce the desired isomorphism.

\bigskip

In a different direction, the properties of the cohomology groups reflect on the geometry of the space, as shown by the following result (proved in \cite{GR17}) which generalizes a classical result of Bochner to the setting of $\RCD$ spaces.
\begin{theorem}
Let $(\X,\sfd,\mm)$ be a $\RCD(0,\infty)$ space. Then $\dim({\sf H}^1_{\sf dR}(\X))\leq \min_\X \dim_{\rm loc}$.

Moreover, if $(\X,\sfd,\mm)$ is $\RCD(0,N)$ and $\dim({\sf H}^1_{\sf dR}(\X))=N$ (so that in particular $N\in\N$), then $\X$ is the flat $N$-dimensional torus.
\end{theorem}
The first part of the statement follows noticing that, much like in the smooth case, harmonic 1-forms must be parallel (because of \eqref{eq:ehec}). The second claim is harder, because the classical proof which passes via universal cover can't be adapted; instead, the desired isomorphism is built from scratch by considering the Regular Lagrangian Flows of a basis of harmonic forms.

\bigskip

In the smooth setting of weighted Riemannian manifolds, it is well known that the validity of a curvature dimension condition is linked to the fact that the $N$-Ricci tensor is bounded from below by $K$ and that $N$ is equal to the geometric dimension of the manifold if and only if the trace of the Hessian is equal to the Laplacian.

Something similar holds on $\RCD^*(K,N)$ spaces, as proved in \cite{Han14} by adapting the computations done in \cite{Sturm14} to the non-smooth setting.  Let us introduce the function $R_N:[H^{1,2}_\Ho(T\X)]^2\to L^1(\X)$ as
\[
R_N(X,Y):=
\dfrac{\big({\rm tr}(\nabla X)-\div X\big)\big({\rm tr}(\nabla Y)-\div Y\big)}{N-\dim_{\rm loc}}\quad\text{ if }\dim_{\rm loc}<N,\qquad\qquad
0\quad\text{ otherwise}
\]
and the $N$-Ricci tensor $\ric_N:[H^{1,2}_\Ho(T\X)]^2\to \mes(\X)$ as
\[
\ric_N(X,Y):=\ric(X,Y)-R_N(X,Y)\mm.
\]
It is easy to see that
\[
\begin{split}
|\nabla X|^2_\HS+R_N(X,X)&\geq\frac{(\div X)^2}N,\qquad\text{ and }\qquad\ric_N(fX,Y)=f\,\ric_N(X,Y),
\end{split}
\]
for every $X,Y\in H^{1,2}_\Ho(T\X)$ and $f\in\test\X$.

The main results in \cite{Han14} can then be summarized as:
\begin{theorem}
Let $(\X,\sfd,\mm)$ be a $\RCD^*(K',\infty)$ space. Then it is a  $\RCD^*(K,N)$ space if and only if
\begin{itemize}
\item[i)] $\dim_{\rm loc}\leq N$ $\mm$-a.e.
\item[ii)] For any $X\in H^{1,2}_\Ho(T\X)$ we have ${\rm tr}(\nabla X)=\div X$ $\mm$-a.e.\ on $\{\dim_{\rm loc}=N\}$
\item[iii)] For any $X\in H^{1,2}_\Ho(T\X)$ we have
\[
\begin{split}
\ric_N(X,X)&\geq K|X|^2\mm.
\end{split}
\]
\end{itemize}
\end{theorem}
\def\cprime{$'$} \def\cprime{$'$}

\end{document}